\documentclass[12pt,a4paper]{amsart}
\usepackage[top=25mm, bottom=25mm, left=27mm, right=27mm]{geometry}
\usepackage{mathptmx}
\usepackage{mathrsfs}
\usepackage{verbatim}
\usepackage{url}
\usepackage[all]{xy}
\usepackage{xcolor}
\usepackage[colorlinks=true,citecolor=blue]{hyperref}
\usepackage{enumitem}
\usepackage{pstricks}
\usepackage{tikz}
\usepackage{setspace}
\newcounter{MainTheoremCounter}

\newtheorem{Maintheorem}[MainTheoremCounter]{Theorem}

\newtheorem{theorem}{Theorem}[section]
\newtheorem{lemma}[theorem]{Lemma}
\newtheorem{ques}[theorem]{Question}

\theoremstyle{definition}
\newtheorem{definition}[theorem]{Definition}

\newtheorem{cor}[theorem]{Corollary}

\newtheorem{claim}[theorem]{Claim}
\newtheorem{case}[theorem]{Case}
\numberwithin{equation}{section}

\newtheorem{prop}[theorem]{Proposition}

\newcommand\blfootnote[1]{
\begingroup
\renewcommand\thefootnote{}\footnote{#1}
\addtocounter{footnote}{-1}
\endgroup
}
\begin{document}

\title{Polynomial orbits in totally minimal systems}

\author[J.~Qiu]{Jiahao Qiu}
\address[J.~Qiu]{School of Mathematical Sciences, Peking University, Beijing, 100871, China}

\email{qiujh@mail.ustc.edu.cn}

\date{\today}

\begin{abstract}
Inspired by the recent work of Glasner, Huang, Shao, Weiss and Ye \cite{GHSWY20},
we prove that the maximal $\infty$-step pro-nilfactor $X_\infty$ of a minimal system $(X,T)$
is the topological characteristic factor along polynomials in a certain sense.
Namely, we show that by an almost one to one modification of $\pi:X\to X_\infty$,
the induced open extension $\pi^*:X^*\to X_\infty^*$ has the following property:
for any $d\in \mathbb{N}$, any open subsets $V_0,V_1,\ldots,V_d$  of $X^*$ with
$\bigcap_{i=0}^d \pi^*(V_i)\neq \emptyset$ and any
distinct non-constant integer polynomials $p_i$ with $p_i(0)=0$ for $i=1,\ldots,d$,
there exists some $n\in \mathbb{Z}$ such that $V_0\cap
T^{-p_1(n)}V_1\cap \ldots \cap T^{-p_d(n)}V_d  \neq \emptyset$.
%where an integer polynomial is the polynomial with rational coefficients
%taking integer values on the integers.

As an application, the following result is obtained:
for a totally minimal system $(X,T)$ and integer polynomials $p_1,\ldots,p_d$,
if every non-trivial integer combination of $p_1,\ldots,p_d$ is not constant,
then there is a dense $G_\delta$ subset $\Omega$ of $ X$ such that
the set
\[
\{(T^{p_1(n)}x,\ldots, T^{p_d(n)}x):n\in \mathbb{Z}\}
\]
is dense in $X^d$
for every $x\in \Omega$.
\end{abstract}
\blfootnote{ \ \ \
This research is supported by NNSF of China (11431012,11971455)
and China Post-doctoral Grant (BX2021014).
}
\maketitle

%\tableofcontents
%\begin{spacing}{2}
	%\tableofcontents
%\end{spacing}   % 若不想要目录, 注释掉该句
%\thispagestyle{empty}   % 不要页眉页脚和页码
%\newpage

%---------以下生成正文----------
%\setcounter{page}{1}    % 从正文开始计页码

\section{Introduction}

In this section, we will provide the background of the research, state the main results of the paper
and give an outline of the ideas for the proofs.

\subsection{Density problems}\

%In ergodic theory there are many results stating that the time averages are equal to the
%%spatial averages under various ergodicity assumptions. For example, the von Neumann
%mean ergodic theorem tells us that if $(X,\mathcal{X} ,\mu,T)$ is ergodic, then for each $f \in L^2(\mu)$,
%one has
%\[
%\lim_{N\to \infty} \frac{1}{N} \sum_{n=0}^{N-1} f(T^{n}x)=\int f\; \mathrm{d}\mu
%\]
%in $L^2$ norm.
%The corresponding topological
%s%tatement is the following: if $(X,T)$ is a transitive system,
%then there is a dense $G_\delta$ subset $\Omega$ of $ X$
%such that each $x\in \Omega$ has a dense orbit.

For a totally ergodic system $(X,\mathcal{X} ,\mu,T)$
(this means $T^k$ is ergodic for any positive integer $k$),
Furstenberg \cite{FH81} shown that for any non-constant integer polynomial
$p$ and $f \in  L^2(\mu)$,
\begin{equation}\label{poly-total-ergodic}
\lim_{N\to \infty} \frac{1}{N} \sum_{n=0}^{N-1} f(T^{p(n)}x)=\int f\; \mathrm{d}\mu
\end{equation}
in $L^2$ norm,
where an integer polynomial is a
polynomial with rational coefficients taking integer values on the integers.
Bourgain \cite{JB89} shown that (\ref{poly-total-ergodic}) holds pointwisely for any
$f \in  L^r(\mu)$ with $r> 1$.

\medskip

For topological dynamics,
the following question is natural.

\begin{ques}\label{Q1}
Let $(X,T)$ be a totally minimal system
(this means $T^k$ is minimal for any positive integer $k$)
and let $p$ be a non-constant integer polymonial.
Is there a point $x\in X$
such that the set $\{T^{p(n)}x : n\in \mathbb{Z}\}$ is dense in $X$?
\end{ques}

Note that for Question \ref{Q1} we cannot use the results of Furstenberg and Bourgain on polynomial convergence for totally ergodic systems,
as not every minimal system admits a totally ergodic measure.
In addition,
the total minimality assumption is necessary,
as can be seen by considering a periodic orbit of period 3 and taking an integer polynomial by $p(n)=n^2$.

%Let $(X,T)$ be a periodic orbit of period 3 and then $(X,T^2)$ is minimal, but it is easy to check that
%$\{T^{n^2}x:n\in \mathbb{Z}\}$ is not dense in $X$ for any $x \in X$.

In \cite{GHSWY20}, it was shown that the answer to Question \ref{Q1} is positive for any integer polynomial of degree 2.

In order to precisely state the equidistribution results for totally ergodic nilsystems
obtained by Frantzikinakis and Kra \cite{FK05},
we start with the following definition.
A family of integer polynomials $\{p_1,\ldots,p_d\}$ is said to be
{\it independent} if for all integers $m_1,\ldots,m_d$ with at least some $m_j\neq 0,j\in\{1,\ldots,d\}$,
the polynomial $\sum_{j=1}^{d}m_jp_j$ is not constant.

%Frantzikinakis and Kra \cite{FK05} proved that

In \cite{FK05},
it was shown that for a totally ergodic nilsystem
which is equivalent to be totally minimal (see for example \cite{LA,Le05A,PW}),
there exists some point whose orbit along an independent family of integer polynomials is
uniformly distributed and thus dense. % in the product space
They also pointed out that the assumption that the polynomial family is independent is necessary,
as can be seen by considering an irrational rotation on the circle.

In this paper, we give an affirmative answer to Question \ref{Q1}.
We prove:

\begin{Maintheorem}\label{poly-orbit}
Let $(X,T)$ be a totally minimal system,
and assume that $\{p_1,\ldots,p_d\}$ is an independent family of integer polynomials.
Then there is a dense $G_\delta$ subset $\Omega$ of $ X$ such that
the set
\begin{equation}\label{product-space}
\{(T^{p_1(n)}x,\ldots, T^{p_d(n)}x):n\in \mathbb{Z}\}
\end{equation}
is dense in $X^d$ for every $x\in \Omega$.
\end{Maintheorem}

If one assumes that $(X,T)$ is minimal and weakly mixing,
Huang, Shao and Ye \cite{HSY19} showed that
for any family of distinct integer polynomials,
the set (\ref{product-space}) is dense.% in the product space.

\subsection{Topological characteristic factors}\

For a measure preserving system $(X,\mathcal{X} ,\mu,T)$ and $f_1,\ldots,f_d\in L^\infty(\mu)$,
the study of convergence of the {\it multiple ergodic averages}
\begin{equation}\label{MEA}
\frac{1}{N} \sum_{n=0}^{N-1}f_1(T^n x)\cdots f_d(T^{dn}x)
\end{equation}
started from Furstenberg's elegant proof of Szemer\'{e}di Theorem \cite{Sz75}
via an ergodic theoretical analysis \cite{FH}.
 %\cite{FH} drew the deep connection between additive combinatorics and ergodic theory,
%showing that
%follows from an ergodic theorem, now known as the multiple recurrence theorem:
%$d\in \mathbb{N}$ and $A\in \mathcal{X}$ with positive measure, one has
%\[
 %\liminf_{N\to \infty} \frac{1}{N}\sum_{n=0}^{N-1}\mu(A\cap T^{-n} A\cap T^{-2n}A\cap \ldots \cap T^{-dn} A)>0.
%\]
%It is natural to ask about the convergence of these averages; or
%more generally, about the convergence, either in $L^2$ norm or pointwise, of the
%\emph{multiple ergodic averages}
%\begin{equation}\label{MEA}
%\frac{1}{N} \sum_{n=0}^{N-1}f_1(T^n x)\cdots f_d(T^{dn}x)
%\end{equation}
%where $f_1,\ldots,f_d\in L^\infty(\mu)$.
After nearly 30 years' efforts of many researchers, this
problem (for $L^2$ convergence) was finally solved in \cite{HK05,TZ07}.
The basic approach is to find an appropriate factor, called a characteristic factor,
that controls the limit behavior in $L^2(\mu)$ of the averages (\ref{MEA}).
For the origin of these ideas and this terminology, see \cite{FH}.
To be more precise, let $(X,\mathcal{X} ,\mu,T)$ be a measure preserving system
and let $(Y,\mathcal{Y},\nu,T)$ be a factor of $X$.
We say that $Y$ is a \emph{characteristic factor} of $X$ if for all
$f_1,\ldots,f_d\in L^\infty(\mu)$,
\[
\lim_{N\to \infty}\frac{1}{N} \sum_{n=0}^{N-1}f_1(T^n x)\cdots  f_d(T^{dn}x)
-\frac{1}{N} \sum_{n=0}^{N-1}\mathbb{E}(f_1| \mathcal{Y})(T^nx)\cdots
\mathbb{E}(f_d| \mathcal{Y})(T^{dn}x)=0
\]
in $L^2$ norm.
The next step is to obtain a concrete description for some well chosen characteristic factor
in order to prove convergence.
The result in \cite{HK05,TZ07} shows that
such a characteristic factor can be described as an inverse limit of nilsystems,
which is also called a pro-nilfactor.

\medskip

The counterpart of characteristic factors for topological dynamics was first studied by
Glasner in \cite{GE94}. To state the result we need a notion called saturated subset.
Given a map $\pi:X\to Y$ of sets $X$ and $Y$, a subset $L$ of $X$ is called $\pi$-\emph{saturated} if
\[
\{  x\in L:\pi^{-1}(\{\pi(x)\})\subset L \}=L
\]
i.e., $L=\pi^{-1}(\pi(L))$.
Here is the definition of topological characteristic factors:
\begin{definition}\cite{GE94}\label{def-saturated}
  Let $\pi:(X,T)\to (Y,T)$ be a factor map of topological dynamical systems and $d\in \mathbb{N}$.
$(Y,T)$ is said to be a \emph{$d$-step topological characteristic factor}
if there exists a
dense $G_\delta$ subset $\Omega$ of $X$ such that for each $x\in \Omega$ the orbit closure
\[
L_x^d(X):=\overline{\mathcal{O}}\big((\underbrace{x,\ldots,x}_{\text{$d$ times}}),T\times \ldots \times T^d\big)
\]
is $\underbrace{\pi\times \ldots\times \pi}_{\text{$d$ times}}$-saturated.
That is, $(x_1,\ldots,x_d)\in L_x^d(X)$ if and only if
$(x_1',\ldots,x_d')\in L_x^d(X)$ whenever $\pi(x_i)=\pi(x_i')$ for every $i=1,\ldots, d$.
%where $x^{(d)}=(x,\ldots,x)$ ($d$ times) and
%$\pi^{(d)}=\pi \times \ldots \times \pi$ ($d$ times).
\end{definition}

In \cite{GE94}, it was shown that for minimal systems, up to a canonically defined proximal
extension, a characteristic family for $T\times \ldots \times T^d$ is the family of canonical PI flows of class $(d-1)$.
In particular, if $(X,T)$ is distal, then its largest class $(d-1)$ distal factor
is its topological characteristic factor along $T\times \ldots \times T^d$. Moreover,
if $(X,T)$ is weakly mixing, then the trivial system is its topological characteristic factor.
For more related results we refer the reader to \cite{GE94}.

\medskip

On the other hand,
to get the corresponding pro-nilfactors for topological dynamics,
in a pioneer work, Host, Kra and Maass \cite{HKM10} introduced the notion of
{\it regionally proximal relation of order $d$}
for a topological dynamical system $(X,T)$, denoted by $\mathbf{RP}^{[d]}(X)$.
For $d\in\mathbb{N}$, we say that a minimal system $(X,T)$ is a \emph{d-step pro-nilsystem}
if $\mathbf{RP}^{[d]}(X)=\Delta$ and this is equivalent for $(X,T)$ being
an inverse limit of minimal $d$-step nilsystems.
For a minimal distal system $(X,T)$, it was proved that
$\mathbf{RP}^{[d]}(X)$ is an equivalence relation and $X/\mathbf{RP}^{[d]}(X)$
is the maximal $d$-step pro-nilfactor \cite{HKM10}.
Later, Shao and Ye \cite{SY12} showed that in fact for
any minimal system, the regionally proximal relation of order $d$ is an equivalence
relation and it also has the so-called lifting property.

\medskip

Very recently,
the result in \cite{GHSWY20} improves Glasner's result to pro-nilsystems significantly.
That is, they proved:

\begin{theorem}[Glasner-Huang-Shao-Weiss-Ye]\label{key-thm0}
Let $(X,T)$ be a minimal system, and let $\pi:X\rightarrow X/\mathbf{RP}^{[\infty]}(X)= X_\infty$ be the factor map.
Then there exist minimal systems $X^*$ and $X_\infty^*$ which are almost one to one
extensions of $X$ and $X_\infty$ respectively, and a commuting diagram below such that  $X_\infty^*$ is a
$d$-step topological characteristic factor of $X^*$ for all $d\ge 2$.
\begin{equation}\label{commuting-diagram}
\xymatrix{
  X \ar[d]_{\pi}  & X^* \ar[d]^{\pi^*}  \ar[l]_{\sigma^*} \\
  X_\infty   & X_\infty^*      \ar[l]_{\tau^*}
  }
\end{equation}
\end{theorem}

\medskip

%From the topological characteristic factor
In the theorem above, one can see that
for any open subsets $V_0,V_1,\ldots,V_d$ of $X^*$ with
$\bigcap_{i=0}^d \pi^*(V_i)\neq \emptyset$,
there is some $n\in \mathbb{Z}$
such that
\[
V_0\cap T^{-n}V_1\cap \ldots\cap T^{-dn} V_d\neq \emptyset.
\]

Based on this result,
in this paper we use PET-induction which was introduced by Bergelson in \cite{BV87},
to give a polynomial version of their work:

\begin{Maintheorem}\label{polynomial-TCF}
Let $(X,T)$ be a minimal system,
and let $\pi:X\to X/\mathbf{RP}^{[\infty]}(X)= X_{\infty}$ be the factor map.
Then there exist minimal systems $X^*$ and $X_\infty^*$
which are almost one to one extensions of $X$ and $X_\infty$ respectively,
and a commuting diagram as in (\ref{commuting-diagram}) such that
for any open subsets $V_0,V_1,\ldots,V_d$ of $X^*$ with
$\bigcap_{i=0}^d \pi^*(V_i)\neq \emptyset$ and any distinct non-constant integer polynomials $p_i$
with $p_i(0)=0$ for $i=1,\ldots,d$,
there exists some $n\in \mathbb{Z}$
such that
\[
V_0\cap T^{-p_1(n)}V_1\cap \ldots\cap T^{-p_d(n)} V_d\neq \emptyset.
\]

\end{Maintheorem}

\subsection{Strategy of the proofs}\

To prove Theorem \ref{poly-orbit},
by Theorem \ref{key-thm0}
it suffices to verify the system $(X^*,T)$
which is also totally minimal.
It is equivalent to prove that
for any given non-empty open subsets $V_0,V_1,\ldots,V_d$ of $X^*$,
there exists some $n\in \mathbb{Z}$
such that
\[
V_0\cap T^{-p_1(n)}V_1\cap \ldots\cap T^{-p_d(n)} V_d\neq \emptyset.
\]

Notice that $X_\infty^*$ is an
almost one to one extension of a totally minimal $\infty$-step pro-nilsystem
which can be approximated arbitrarily
well by a nilsystem (see \cite[Theorem 3.6]{DDMSY13}),
we get that $( X_{\infty}^*,T)$ satisfies Theorem \ref{poly-orbit} (Lemma \ref{equi-condition}),
which implies there is some $m\in \mathbb{Z}$ such that
\[
\pi^*(V_0)\cap T^{-p_1(m)}\pi^*(V_1)\cap\ldots \cap T^{-p_d(m)}\pi^*(V_d)\neq \emptyset.
\]

Using Theorem \ref{polynomial-TCF}
for open sets $V_0,T^{-p_1(m)}V_1,\ldots,T^{-p_d(m)}V_d$
and integer polynomials $p_1(\cdot+m)-p_1(m),\ldots,p_d(\cdot+m)-p_d(m)$,
there is some $k\in \mathbb{Z}$ such that
\[
V_0\cap T^{-p_1(k+m)}V_1\cap \ldots\cap T^{-p_d(k+m)} V_d\neq \emptyset,
\]
as was to be shown.

\medskip

To prove Theorem \ref{polynomial-TCF},
we use PET-induction, which was introduced by Bergelson in \cite{BV87},
where PET stands for {\it polynomial ergodic theorem} or
{\it polynomial exhaustion technique} (see \cite{BV87,BM00}).
See also \cite{BL96,BL99} for more on PET-induction.

Basically, we associate any finite collection of integer polynomials with a `complexity',
and reduce the complexity at some step to the trivial one.
Note that in some step,
the cardinal number of the collection may increase while the complexity decreases.

We note that when doing the induction procedure,
we find that the known results are not enough to
guarantee them,
and we need a stronger result (Theorem \ref{polynomial-case}).
After we introduce PET-induction in Subsection \ref{PET-induction-def}, we will explain the main ideas for the
proof via proving some simple examples.

\subsection{The organization of the paper}\

The paper is organized as follows.
In Section \ref{pre}, the basic notions used in the paper are introduced.
In Section \ref{pf-thm-A}, we first give a proof of Theorem \ref{poly-orbit} assuming Theorem \ref{polynomial-TCF}
whose proof is very complicated.
In Section \ref{pf-thm-B}, we prove Theorem \ref{polynomial-TCF}.

\bigskip

\noindent {\bf Acknowledgments.}
The author would like to thank Professors Wen Huang, Song Shao and Xiangdong Ye for helping discussions.
%The author was supported by NNSF of China (11801538,11971455)
%and China Post-doctoral Grant (BX2021014).

\section{Preliminaries}\label{pre}
In this section we gather definitions and preliminary results that
will be necessary later on.
Let $\mathbb{N}$ and $\mathbb{Z}$ be the sets of all positive integers
and integers respectively.

\subsection{Topological dynamical systems}\

A \emph{topological dynamical system}
 (or \emph{system}) is a pair $(X,T)$,
 where $X$ is a compact metric space with a metric $\rho$, and $T : X \to  X$
is a homeomorphism.
%If $A$ is a non-empty closed subset of $X$ and $TA\subset A$, then $(A,T|_A)$ is called a \emph{subsystem} of $(X,T)$,
%where $T|_A$ is the restriction of $T$ on $A$.
%If there is no ambiguity, we use the notation $T$ instead of $T|_A$.
For $x\in X$, the \emph{orbit} of $x$ is given by $\mathcal{O}(x,T)=\{T^nx: n\in \mathbb{Z}\}$.
For convenience, we denote the orbit closure of $x\in X$
under $T$ by $\overline{\mathcal{O}}(x,T)$,
instead of $\overline{\mathcal{O}(x,T)}$.
A system $(X,T)$ is said to be \emph{minimal} if
every point has a dense orbit,
and \emph{totally minimal} if $(X,T^k)$ is minimal for
any positive integer $k$.

\medskip

A \emph{homomorphism} between systems $(X,T)$ and $(Y,T)$ is a continuous onto map
$\pi:X\to Y$ which intertwines the actions; one says that $(Y,T)$ is a \emph{factor} of $(X,T)$
and that $(X,T)$ is an \emph{extension} of $(Y,T)$. One also refers to $\pi$ as a \emph{factor map} or
an \emph{extension} and one uses the notation $\pi : (X,T) \to (Y,T)$.
%The systems are said to be \emph{conjugate} if $\pi$ is a bijection.
An extension $\pi$ is determined
by the corresponding closed invariant equivalence relation $R_\pi=\{(x,x')\in X\times X\colon \pi(x)=\pi(x')\}$.

\subsection{Regional proximality of higher order}\

For $\vec{n} = (n_1,\ldots, n_d)\in \mathbb{Z}^d$ and $\epsilon=(\epsilon_1,\ldots,\epsilon_d)\in \{0,1\}^d$, we
define
$\displaystyle\vec{n}\cdot \epsilon = \sum_{i=1}^d n_i\epsilon_i $.

\begin{definition}\cite{HKM10}\label{definition of pronilsystem and pronilfactor}
Let $(X,T)$ be a system and $d\in \mathbb{N}$.
The \emph{regionally proximal relation of order $d$} is the relation $\mathbf{RP}^{[d]}(X)$
defined by: $(x,y)\in\textbf{RP}^{[d]}(X)$ if
and only if for any $\delta>0$, there
exist $x',y'\in X$ and
$\vec{n}\in \mathbb{N}^d$ such that:
$\rho(x,x')<\delta,\rho(y,y')<\delta$ and
\[
\rho(  T^{\vec{n}\cdot\epsilon} x', T^{\vec{n}\cdot\epsilon}  y')<\delta,
\quad \forall\;\epsilon\in \{0,1\}^d\backslash\{ \vec{0}\}.
\]

A system is called a \emph{$d$-step pro-nilsystem}
if its regionally proximal relation of order $d$ is trivial,
i.e., coincides with the diagonal.
\end{definition}

It is clear that for any system $(X,T)$,
\[
\ldots\subset \mathbf{RP}^{[d+1]}(X)\subset \mathbf{RP}^{[d]}(X)\subset
\ldots \subset\mathbf{RP}^{[2]}(X)\subset \mathbf{RP}^{[1]}(X).
\]

\begin{theorem}\cite[Theorem 3.3]{SY12}\label{cube-minimal}
For any minimal system and $d\in \mathbb{N}$,
the regionally proximal relation of order $d$
 is a closed invariant equivalence relation.
\end{theorem}

It follows that for any minimal system $(X,T)$,
\[
\mathbf{RP}^{[\infty]}(X)=\bigcap_{d\geq1}\mathbf{RP}^{[d]}(X)
\]
is also a closed invariant equivalence relation.

Now we formulate the definition of $\infty$-step pro-nilsystems.

\begin{definition}
A minimal system $(X,T)$ is an \emph{$\infty$-step pro-nilsystem},
if the equivalence relation $\mathbf{RP}^{[\infty]}(X)$ is trivial,
i.e., coincides with the diagonal.
\end{definition}

The regionally proximal relation of order $d$ allows us to construct the \emph{maximal $d$-step pro-nilfactor}
of a minimal system. That is, any factor of $d$-step pro-nilsystem
factorizes through this system.

\begin{theorem}\label{lift-property}\cite[Theorem 3.8]{SY12}
Let $\pi :(X,T)\to (Y,T)$ be a factor map of minimal systems and $d\in \mathbb{N}\cup\{\infty\}$. Then,
\begin{enumerate}
\item $(\pi \times \pi) \mathbf{RP}^{[d]}(X)=\mathbf{RP}^{[d]}(Y)$.
\item $(Y,T)$ is a $d$-step pro-nilsystem if and only if $\mathbf{RP}^{[d]}(X)\subset R_\pi$.
\end{enumerate}

In particular, the quotient of $(X,T)$ under $\mathbf{RP}^{[d]}(X)$
is the maximal $d$-step pro-nilfactor of $X$.
\end{theorem}

\subsection{Nilpotent groups, nilmanifolds and nilsystems}\

Let $L$ be a group.
For $g,h\in L$, we write $[g,h]=ghg^{-1}h^{-1}$ for the commutator of $g$ and $h$,
we write $[A,B]$ for the subgroup spanned by $\{[a,b]:a\in A,b\in B\}$.
The commutator subgroups $L_j,j\geq1$, are defined inductively by setting $L_1=L$
and $L_{j+1}=[L_j,L],j\geq1$.
Let $k\geq 1$ be an integer.
We say that $L$ is \emph{k-step nilpotent} if $L_{k+1}$ is the trivial subgroup.

Let $L$ be a $k$-step nilpotent Lie group and $\Gamma$ be a discrete cocompact subgroup of $L$.
The compact manifold $X=L/\Gamma$ is called a \emph{k-step nilmanifold.}
The group $L$ acts on $X$ by left translation and we write this action as $(g,x)\mapsto gx$.
Let $\tau\in L$ and $T$ be the transformation $x\mapsto \tau x$ of $X$.
Then $(X,T)$ is called a \emph{k-step nilsystem}.

We also make use of inverse limits of nilsystems and so we recall the definition of an inverse limit
of systems (restricting ourselves to the case of sequential inverse limits).
If $\{(X_i,T_i)\}_{i\in \mathbb{N}}$ are systems with $\text{diam}(X_i)\leq 1$ and $\phi_i:X_{i+1}\to X_i$
are factor maps, the \emph{inverse limit} of the systems is defined to be the compact subset of
$\prod_{i\in \mathbb{N}}X_i$
given by $\{(x_i)_{i\in\mathbb{N}}:\phi_i(x_{i+1})=x_i,i\in \mathbb{N}\}$,
which is denoted by $\lim\limits_{\longleftarrow}\{ X_i\}_{i\in \mathbb{N}}$.
It is a compact metric space endowed with the
distance $\rho(x,y)=\sum_{i\in \mathbb{N}}1/ 2^i \rho_i(x_i,y_i)$.
We note that the maps $\{T_i\}_{i\in \mathbb{N}}$ induce a transformation $T$ on the inverse limit.

\medskip

The following structure theorems characterize inverse limits of nilsystems.

\begin{theorem}[Host-Kra-Maass]\cite[Theorem 1.2]{HKM10}\label{description}
Let $d\geq2$ be an integer.
A minimal system is a $d$-step pro-nilsystem
if and only if
it is an inverse limit of minimal $d$-step nilsystems.
\end{theorem}

\begin{theorem}\cite[Theorem 3.6]{DDMSY13}\label{system-of-order-infi}
A minimal system is an $\infty$-step pro-nilsystem
if and only if it is an inverse limit of minimal nilsystems.
\end{theorem}

\subsection{Some facts about hyperspaces and fundamental extensions}\

Let $X$ be a compact metric space with a metric $\rho$.
Let $2^X$ be the collection of non-empty closed subsets of $X$.
We may define a metric on $2^X$ as follows:
\begin{align*}
\rho_H(A,C) &= \inf\{  \epsilon>0:A\subset B(C,\epsilon),C\subset B(A,\epsilon) \} \\
& =\max\{ \max_{a\in A} \rho(a,C),\max_{c\in C} \rho(c,A)\},
\end{align*}
where $\rho(x,A)=\inf_{y\in A}\rho(x,y)$ and $B(A,\epsilon)=\{x\in X:\rho(x,A)<\epsilon\}$.
The metric $\rho_H$ is called the \emph{Hausdorff metric} on $2^X$.

\medskip

Let $\pi:(X,T)\to (Y,T)$ be a factor map of systems.
We say that:
\begin{enumerate}
\item $\pi$ is an \emph{open} extension if it is open as a map;
\item  $\pi$ is an \emph{almost one to one} extension if there is a dense $G_\delta$ subset
$\Omega$ of $X$ such that $\pi^{-1}(\{\pi (x)\})=\{x\}$ for every $x\in \Omega$.
\end{enumerate}

The following is a well known fact about open mappings (see \cite[Appendix A.8]{JDV} for example).

\begin{theorem}\label{open-map}
Let $\pi:(X,T)\to (Y,T)$ be a factor map of systems.
Then the map $\pi^{-1}:Y\to 2^X,y \mapsto  \pi^{-1}(y)$ is continuous
if and only if $\pi$ is open.
\end{theorem}

\subsection{Polynomial orbits in minimal systems}\

We have the following characterization of polynomial orbits in minimal systems.

\begin{lemma}\label{equi-condition}
Let $(X,T)$ be a minimal system and let $p_1,\ldots,p_d$ be non-constant integer polynomials.
Then the following statements are equivalent:
\begin{enumerate}
\item There exists a dense $G_\delta$ subset $\Omega$
of $X$ such that
the set
\[
\{(T^{p_1(n)}x,\ldots, T^{p_d(n)}x):n\in \mathbb{Z}\}
\]
is dense in $X^d$ for every $x\in \Omega$.
\item There exists some $x\in X$ such that the set
\[
\{(T^{p_1(n)}x,\ldots, T^{p_d(n)}x):n\in \mathbb{Z}\}
\]
is dense in $X^d$.
\item For any non-empty open subsets $U,V_1,\ldots,V_d$ of $X$,
there is some $n\in \mathbb{Z}$ such that
\[
U\cap T^{-p_1(n)}V_1\cap\ldots \cap T^{-p_d(n)}V_d\neq \emptyset.
\]
\end{enumerate}
\end{lemma}

\begin{proof}
(1) $\Rightarrow$ (2) is obvious.

(2) $\Rightarrow$ (3):
Assume there is some $x\in X$ such that
the set
\[
\{(T^{p_1(n)}x,\ldots, T^{p_d(n)}x):n\in \mathbb{Z}\}
\]
is dense in $X^d$.
It is clear that for any $m\in \mathbb{Z}$, the set
\[
X(x,m):= \{\big(T^{p_1(n)}(T^mx),\ldots, T^{p_d(n)}(T^mx)\big):n\in \mathbb{Z}\}
\]
is dense in $X^d$.

Fix non-empty open subsets $U,V_1,\ldots,V_d$ of $X$.
As $(X,T)$ is minimal, there is some $m\in \mathbb{N}$ with $T^mx\in U$.
Notice $ X(x,m)$ is dense in $X^d$,
we can choose some $n\in \mathbb{Z}$ such that
$T^{p_i(n)}(T^mx)\in V_i$ for $i=1,\ldots,d$, which implies
\[
T^mx\in U\cap T^{-p_1(n)}V_1\cap\ldots \cap T^{-p_d(n)}V_d.
\]

(3) $\Rightarrow$ (1):
Assume that for any given non-empty open subsets $U,V_1,\ldots,V_d$ of $X$,
there is some $n\in \mathbb{Z}$ such that
$U\cap T^{-p_1(n)}V_1\cap\ldots \cap T^{-p_d(n)}V_d\neq \emptyset.$

Let $\mathcal{F}$ be a countable base of $X$, and let
\[
\Omega:=\bigcap_{V_1,\ldots,V_d\in \mathcal{F}} \bigcup_{n\in \mathbb{Z}}
T^{-p_1(n)}V_1\cap\ldots \cap T^{-p_d(n)}V_d.
\]
Then it is easy to see that the dense $G_\delta$ subset $\Omega$ is what we need.
\end{proof}

The following result can be derived from
%the Frantzikinakis and Kra's equidistribution results for totally minimal nilsystems
\cite[Theorem 1.2]{FK05}.

\begin{cor}\label{uniform-dis}
Let $(X=L/\Gamma,T)$ be a totally minimal nilsystem,
and assume that $\{p_1,\ldots,p_d\}$ is an independent family of integer polynomials.
Then there is some $x\in X$ such that the set
\[
\{ (T^{p_1(n)}x,\ldots,T^{p_d(n)}x):n\in \mathbb{Z}\}
\]
is dense in $X^d$.
\end{cor}

By Theorem \ref{system-of-order-infi} and Corollary \ref{uniform-dis} we have:\

\begin{cor}\label{uni-distributed-AA}
Let $(X,T)$ be a totally minimal system
and assume that $\{p_1,\ldots,p_d\}$ is an independent family of integer polynomials.
If $(X,T)$ is an almost one to one extension of an $\infty$-step pro-nilsystem,
then there is some $x\in X$ such that the set
\[
\{(T^{p_1(n)}x,\ldots, T^{p_d(n)}x):n\in \mathbb{Z}\}
\]
is dense in $X^d$.
\end{cor}

\subsection{Polynomial recurrence}\

\medskip

Recall that a collection $\mathcal{F}$ of subsets of $\mathbb{Z}$ is a \emph{family}
if it is hereditary upward, i.e.,
$F_1 \subset F_2$ and $F_1 \in \mathcal{F}$ imply $F_2 \in \mathcal{F}$.
A family $\mathcal{F}$ is called \emph{proper} if it is neither empty nor the entire power set of $\mathbb{Z}$,
or, equivalently if $\mathbb{Z}\in \mathcal{F}$ and $\emptyset\in \mathcal{F}$.
For a family $\mathcal{F}$ its \emph{dual} is the family
$\mathcal{F}^*:=\{   F\subset \mathbb{Z}: F\cap F' \neq \emptyset \;\mathrm{for} \; \mathrm{all}\; F'\in \mathcal{F} \} $.
It is not hard to see that
$\mathcal{F}^*=\{F\subset \mathbb{Z}:\mathbb{Z}\backslash F\notin \mathcal{F}\}$, from which we have that if $\mathcal{F}$ is a family then $(\mathcal{F}^*)^*=\mathcal{F}$.
If a family $\mathcal{F}$ is closed under finite intersections and is proper, then it is called a \emph{filter}.
A family $\mathcal{F}$ has the {\it Ramsey property} if $A = A_1\cup A_2 \in \mathcal{F}$ implies that $A_1 \in \mathcal{F}$
or $A_2 \in F$. It is well known that a proper family has the Ramsey property if and
only if its dual $\mathcal{F}^*$ is a filter \cite{FH}.

\medskip

%Let $\{b_i\}_{i\in I}$ be a finite or infinite sequence in $\mathbb{Z}$.
%One defines
%\[
%FS(\{b_i\}_{i\in I}):=\big\{\sum_{i\in \alpha}b_i:\alpha\;
%\text{is a finite non-empty subset of}\;\; I\big\} \backslash \{0\}.
%\]

For $j\in \mathbb{N}$ and a finite subset $\{p_1, \ldots, p_j\}$ of $\mathbb{Z}$, the
\emph{finite IP-set of length $j$} generated by $\{p_{1}, \ldots, p_j\}$
is the set
\[
\big\{p_{1}\epsilon_{1}+ \ldots+ p_j\epsilon_j: \epsilon_1,\ldots,\epsilon_j\in \{0,1\}\big\} \backslash \{0\}.
\]
The collection of all sets containing finite IP-sets with arbitrarily long lengths is denoted by $\mathcal{F}_{fip}$.

\begin{lemma}\cite[Lemma 8.1.6]{HSY16}
$\mathcal{F}_{fip}$ has the Ramsey property.
\end{lemma}

Then we have:

\begin{cor}\label{filter}
$\mathcal{F}_{fip}^*$ is a filter.
\end{cor}

\medskip

For a system $(X,T)$, $x\in X$, a non-constant integer polynomial $p$ and a non-empty open subset $V$ of $X$, set
\[
N(x,V)=\{ n\in \mathbb{Z}:T^nx\in V\}\quad
\mathrm{and}\quad
N_p(x,V)=\{n\in \mathbb{Z}:T^{p(n)}x\in V\}.
\]

\begin{prop}\cite[Proposition 8.1.5]{HSY16}\label{RP-infi}
Let $(X,T)$ be a minimal system and $(x,y)\in X\times X\backslash \Delta$.
Then $(x,y)\in \mathbf{RP}^{[\infty]}(X)$ if and only if $N(x,V)\in \mathcal{F}_{fip}$
for every open neighborhood $V$ of $y$.
\end{prop}

%The following result can be derived from
%\cite[Theorem 0.2]{BL18}
% and also

 The following proposition follows from the argument in the proof of \cite[Theorem 8.1.7]{HSY16},
 which also can be derived from \cite[Theorem 0.2]{BL18}.
%we can show:
%The following result
%provides a characterization of pre-nilsystems in terms of IP
%0-recurrence

\begin{prop}\label{fip-family}
Let $(X,T)$ be a minimal $\infty$-step pro-nilsystem.
Then for any $x\in X$, $N(x,V)\in \mathcal{F}_{fip}^*$
for every open neighbourhood $V$ of $x$.
\end{prop}

The following proposition is from \cite[Section 2.11]{Le05B}.

\begin{prop}\label{poly-in-nilsystem}
Let $(X,T)$ be a nilsystem, $x\in X$ and an open neighborhood $U$ of $x$.
For any non-constant integer polynomial $p$ with $p(0)=0$,
we can find another `larger' nilsystem $(Y, S)$ with $y\in Y$ and
 an open neighborhood $V$ of $y$ such that
\[
\{n\in \mathbb{Z}:T^{p(n)}x\in U\}\supset \{n\in \mathbb{Z}:S^ny\in V\}.
\]
\end{prop}

It follows Theorem \ref{system-of-order-infi} that
a minimal $\infty$-step pro-nilsystem is an inverse limit of minimal nilsystems.
By Propositions \ref{fip-family} and \ref{poly-in-nilsystem}
we can get:

\begin{prop}\label{infi-poly-rec}
Let $(X,T)$ be a minimal $\infty$-step pro-nilsystem
and let $p$ be a non-constant integer polynomial with $p(0)=0$.
Then for any $x\in X$,
$N_{p}(x,V)\in \mathcal{F}_{fip}^*$
for every open neighbourhood $V$ of $x$.
\end{prop}

Notice that $\mathcal{F}_{fip}^*$ is a filter,
thus by Proposition \ref{infi-poly-rec} we have:

\begin{cor}\label{return-time-AA}
Let $(X,T)$ be an almost one to one extension of a minimal $\infty$-step pro-nilsystem
and let $p_1,\ldots,p_d$ be non-constant integer polynomials with $p_i(0)=0$ for $i=1,\ldots,d$.
Then there exists a dense $G_\delta$ subset $\Omega$ of $X$
such that for any $x\in \Omega$,
$\bigcap_{i=1}^dN_{p_i}(x,V)\in \mathcal{F}_{fip}^*$
for every open neighbourhood $V$ of $x$.
\end{cor}

\subsection{A useful lemma}\

To end this section we give a useful lemma which can be derived from
the proof of \cite[Theorem 5.6]{GHSWY20}.
For completeness, we include the proof here.
To do this, we need the following topological characteristic factor theorem.

\begin{theorem}\cite[Theorem 4.2]{GHSWY20}\label{key-thm}
Let $\pi:(X,T)\to (Y,T)$ be a factor map of minimal systems.
If $\pi$ is open and $X/ \mathbf{RP}^{[\infty]}(X)$ is a factor of $Y$,
then $Y$ is a $d$-step topological characteristic factor of $X$ for all $d\in \mathbb{N}$.
%That is, for all $d\in \mathbb{N}$
%there is a dense $G_\delta$ subset $\Omega_d$ of $X$ such that for each $x\in \Omega_d$,
%$L_x^d(X)=\overline{\mathcal{O}}(x^{(d)}, T\times \ldots \times T^d)$
%is $\pi^{(d)}$-saturated.
\end{theorem}

With the help of the above powerful theorem we are able to show:

\begin{lemma}\label{fip-infi-fiber}
Let $\pi:(X,T)\to (Y,T)$ be a factor map of minimal systems.
If $\pi$ is open and $X/ \mathbf{RP}^{[\infty]}(X)$ is a factor of $Y$,
then for any distinct positive integers $a_1,\ldots,a_s$,
there is a dense $G_\delta$ subset $\Omega$ of $X$ such that for
any open subsets $V_0,V_1,\ldots,V_s$ of $X$ with
$\bigcap_{i=0}^s \pi(V_i)\neq \emptyset$ and any $z\in V_0 \cap \Omega$ with $\pi(z)\in \bigcap_{i=0}^s \pi(V_i)$,
there exists some $A\in \mathcal{F}_{fip}$
such that $T^{a_in}z\in V_i$ for every $i=1,\ldots,s$ and $n\in A$.
\end{lemma}

\begin{proof}
By Theorem \ref{key-thm},
for every $d\in \mathbb{N}$
there is a dense $G_\delta$ subset $\Omega_d$ of $X$ such that for each $x\in \Omega_d$,
$L_x^d(X)=\overline{\mathcal{O}}((x,\ldots,x),T\times \ldots \times T^d)$
is $\pi\times\ldots\times \pi$-saturated.\footnote{See Definition \ref{def-saturated}.}
Set $\Omega=\bigcap_{d\in \mathbb{N}}\Omega_d$,
then $\Omega$ is a dense $G_\delta$ subset of $X$.

We next show that the set $\Omega$ meets our requirement.

Now fix distinct positive integers $a_1,\ldots,a_s$.
Let $V_0,V_1,\ldots,V_s$ be open subsets of $X$ with
$W:=\bigcap_{i=0}^s \pi(V_i)\neq \emptyset$,
then $\pi^{-1}(W)\cap V_i\neq \emptyset$ for every $i=0,1,\ldots,s$.
Let $z\in \Omega \cap V_0 \cap \pi^{-1}(W)$.
For $i=1,\ldots,s$, let $z_i\in \pi^{-1}(\{  \pi(z) \} )\cap V_i$  and choose $\delta>0$ with $B(z_i,\delta)\subset V_i$.

Set $N=\max\{a_1,\ldots,a_s\}$.
Let $\{b_j\}_{j\in \mathbb{N}}$ be a sequence of positive integers such that $b_{j+1}\geq N(b_1+\ldots+b_j)+1$,
and let $I_j$ be the finite IP-set generated by $\{b_1,\ldots,b_j\}$.
%i.e., $I_j=\big\{b_{1}\epsilon_{1}+ \ldots+ b_j\epsilon_j: \epsilon_1,\ldots,\epsilon_j\in \{0,1\}\big\} \backslash \{0\}$.

\medskip

\noindent {\bf Claim:}
For $i,i'\in\{1,\ldots,N\}$ and $m,m'\in I_j$,
$im=i'm'$ if and only if $i=i'$ and $m=m'$.

\begin{proof}[Proof of Claim]
Suppose for a contradiction that there exist
$i,i'\in\{1,\ldots,N\}$ with $i<i'$ and $m,m'\in I_j$ such that $im=i'm'$.
Let $\epsilon_1,\ldots,\epsilon_j,\epsilon_1',\ldots,\epsilon_j'\in \{0,1\}$
such that $m=b_1\epsilon_1+\ldots+b_j\epsilon_j$ and $m'=b_1\epsilon_1'+\ldots+b_j\epsilon_j'$.
Let
\[
j_0=\max\{ 1\leq n \leq j: \epsilon_n+\epsilon_n'>0\}.
\]

If $j_0=1$, then $m=m'=b_1$,
and thus $im< i'm'$.

If $j_0\geq 2$, then we have
\begin{align*}
b_{j_0}\leq |i\epsilon_{j_0}-i'\epsilon_{j_0}'|b_{j_0} &=\big|i\sum_{n=1}^{j_0-1}b_n\epsilon_n-
i'\sum_{n=1}^{j_0-1}b_n\epsilon_n'\big| \\
& \leq \sum_{n=1}^{j_0-1}b_n|i\epsilon_n-i'\epsilon'_n|\\
& \leq N(b_1+\ldots+b_{j_0-1}),
\end{align*}
which is a contradiction
by the choice of $b_{j_0}$.
This shows the claim.
\end{proof}

For $k\in \mathbb{N}$, let $B_k=b_1+\ldots+b_k$ and
let $\vec{z}_k=(z_1^{k},\ldots,z_{B_kN}^k)\in X^{B_k N} $
such that
\[
z_{j}^k=
\begin{cases}
z_i, & j=a_im,\;\mathrm{where}\; i\in \{1,\ldots,s\},\; m\in I_k;\\
z, &   \hbox{otherwise.}
\end{cases}
\]
By the claim above, every $\vec{z}_k$ is well defined.

Recall that for any $d\in \mathbb{N}$,
$L_z^d(X)$ %=\overline{\mathcal{O}}((z,\ldots,z),T\times \ldots \times T^d)$
is $\pi\times \ldots\times \pi$-saturated,
then $\vec{y}\in L_z^d(X)$ for any $\vec{y}=(y_1,\ldots,y_d)\in X^d$ with
$\pi(z)=\pi(y_i)$ for $i=1,\ldots,d$.
In particular, $\vec{z}_k\in L_z^{B_kN}(X)$
which implies that there is some $n_k\in \mathbb{N}$
such that
$\rho(T^{jn_k}z,z_j^k)<\delta$ for $j=1,\ldots,B_kN$.

Let $A_k=\{m n_k :m\in I_k\}$ and $A=\bigcup_{k\in \mathbb{N}}A_k$.
Then $A\in \mathcal{F}_{fip}$ and
$T^{a_in}z\in B(z_i,\delta)\subset V_i$ for $i=1,\ldots,s$ and $n\in A$.

This completes the proof.
\end{proof}

\section{Proof of Theorem \ref{poly-orbit} assuming Theorem \ref{polynomial-TCF}}\label{pf-thm-A}

In this section, assuming Theorem \ref{polynomial-TCF} we give a proof of Theorem \ref{poly-orbit}.
We start with the following simple observation.

\begin{lemma}\label{total-minimality-proximal}
Let $\pi:(X,T)\to (Y,T)$ be an almost one to one extension of minimal systems.
Then $(X,T)$ is totally minimal if and only if $(Y,T)$ is totally minimal.
\end{lemma}

Now we are in position to show Theorem \ref{poly-orbit} assuming Theorem \ref{polynomial-TCF}.

\begin{proof}[Proof of Theorem \ref{poly-orbit} assuming Theorem \ref{polynomial-TCF}]
Let $(X,T)$ be a totally minimal system
and let $X_\infty=X/\mathbf{RP}^{[\infty]}(X)$.
It follows from Theorem \ref{polynomial-TCF} that there exist minimal systems $X^*$ and $X_\infty^*$
which are almost one to one extensions of $X$ and $X_\infty$ respectively,
and a commuting diagram below:
\[
\xymatrix{
  X \ar[d]_{\pi}  & X^* \ar[d]^{\pi^*}  \ar[l]_{\sigma^*} \\
  X_\infty   & X_\infty^*      \ar[l]_{\tau^*}
  }
\]

By Lemma \ref{total-minimality-proximal}, $(X^*,T)$ and $(X_\infty^*,T)$ are both totally minimal.
It suffices to verify Theorem \ref{poly-orbit} for system $(X^*,T)$.

Let $\{p_1,\ldots,p_d\}$ be an independent family of integer polynomials,
and let $V_0,V_1,\ldots,V_d$ be non-empty open subsets of $X^*$.
As $\pi^*$ is open, $\pi^*(V_0),\pi^*(V_1),\ldots,\pi^*(V_d)$ are non-empty open subsets of $X^*_\infty$.
Notice that $X^*_\infty$ is an almost one to one extension of $X_\infty$
which is a minimal $\infty$-step pro-nilsystem,
thus by Lemma \ref{equi-condition} and Corollary  \ref{uni-distributed-AA},
there is some $m\in \mathbb{N}$ such that
\begin{align*}
   & \pi^*(V_0)\cap\pi^*( T^{-p_1(m)}V_1)\cap\ldots \cap \pi^*( T^{-p_d(m)}V_d) \\
  = & \pi^*(V_0)\cap T^{-p_1(m)}\pi^*(V_1)\cap\ldots \cap T^{-p_d(m)}\pi^*(V_d)\neq \emptyset
\end{align*}

For $i=1,\ldots,d$, let $p_i'(n)=p_i(n+m)-p_i(m)$.
Then every $p_i'$ is an integer polynomial with $p_i'(0)=0$.
Now using Theorem \ref{polynomial-TCF} for integer polynomials $p_1',\ldots,p_d'$
and open sets $V_0, T^{-p_1(m)}V_1,\ldots,T^{-p_d(m)}V_d$,
there exists some $k\in \mathbb{N}$ such that
\begin{align*}
   & V_0\cap T^{-p_1(k+m)}V_1\cap\ldots \cap T^{-p_d(k+m)}V_d \\
  = & V_0\cap T^{-p_1'(k)}(T^{-p_1(m)}V_1)\cap\ldots \cap  T^{-p_d'(k)}( T^{-p_d(m)}V_d)\neq \emptyset,
\end{align*}
which implies Theorem \ref{poly-orbit} for system $(X^*,T)$ by Lemma \ref{equi-condition}.

This completes the proof.
\end{proof}

\section{Proof of Theorem \ref{polynomial-TCF}}\label{pf-thm-B}

In this section, we will prove Theorem \ref{polynomial-TCF}.
%Recall that an integer polynomial
%is the polynomial with rational coefficients
%taking integer values on the integers.
Let $\mathcal{P}^*$ be the set of all non-constant integer polynomials
{\bf taking zero value at zero}.
A \emph{system} $\mathcal{A}$ is a finite subset of $\mathcal{P}^*$.

\subsection{The PET-induction}\label{PET-induction-def}\

Two integer polynomials $p,q$ will be called \emph{equivalent}
if they have the same degree and the leading coefficients
of the polynomials $p,q$ coincide as well.
If $C$ is a set of equivalent integer polynomials,
its \emph{degree} $w(C)$ is the degree of any its members.

For every system $\mathcal{A}$, we define its \emph{weight vector} $\phi(\mathcal{A})$
as follows.
Let $w_1<\ldots<w_k$ be the
set of the distinct degrees of all equivalence classes appeared in $\mathcal{A}$.
For $1\leq i\leq k$, let $\phi(w_i)$ be the number of the equivalence classes of elements of $\mathcal{A}$
with degree $w_i$. Let the weight vector $\phi(\mathcal{A})$ be
\[
\phi(\mathcal{A})=\big((\phi(w_1),w_1),\ldots,(\phi(w_k),w_k)\big).
\]
%If $k=1$, we just write $\phi(\mathcal{A})=(\phi(w_1),w_1)$.

\medskip
For example, the weight vector of $\{c_1n,\ldots,c_sn\}$ is $(s,1)$ if $c_1,\ldots,c_s$
are distinct non-zero integers;
the weight vector of $\{an^2+b_1n,\ldots,an^2+b_tn\}$ ($a$ is a non-zero integer) is $(1,2)$;
and
the weight vector of $\{an^2+b_1n,\ldots,an^2+b_tn,\;c_1n,\ldots,c_sn\}$
($a$ is a non-zero integer and $c_1,\ldots,c_s$
are distinct non-zero integers) is $\big((s,1),(1,2)\big)$;
and the weight vector of the general polynomials of degree not more than 2
is $\big((s,1),(k,2)\big)$.

Let $\mathcal{A},\mathcal{A}'$ be two systems.
We say that $\mathcal{A}'$ \emph{precedes} $\mathcal{A}$
if there exists a degree $w$ such that $\phi(\mathcal{A}')(w)<\phi(\mathcal{A})(w)$ and
$\phi(\mathcal{A})(u)=\phi(\mathcal{A}')(u)$ for any degree $u>w$.
We denote it by $\phi(\mathcal{A}')\prec \phi(\mathcal{A})$.
Under the order of weight vectors, we have
\begin{align*}
&(1,1)\prec (2,1)\prec\ldots\prec (m,1)\prec\ldots \prec(1,2)\prec\big((1,1),(1,2)\big)\prec\ldots \prec\\
& \big((m,1),(1,2)\big)\prec\ldots\prec(2,2)\prec\big((1,1),(2,2)\big)\prec\ldots \prec \big((m,1),(2,2)\big)\prec\ldots \prec\\
&\big((m,1),(k,2)\big)\prec\ldots \prec(1,3)\prec\big((1,1),(1,3)\big)\prec\ldots
\big((m,1),(k,2),(1,3)\big)\prec\ldots \prec\\
& (2,3)\prec\ldots \prec
\big((a_1,1),(a_2,2)\ldots,(a_k,k)\big)\prec\ldots.
\end{align*}

For $p\in \mathcal{P}^*$ and $m\in \mathbb{Z}$,
define $(\partial_m p)(n):=p(n+m)-p(m)$.
It is clear that $\partial_m p \in \mathcal{P}^*$ for any $p\in \mathcal{P}^*$ and $m\in \mathbb{Z}$.

The following lemma can be found in \cite{BL96,Le94}.
\begin{lemma}\label{PET-induction}
Let $\mathcal{A}$ be a system and let $m_1,\ldots,m_d $ be distinct non-zero integers.
Let $p\in \mathcal{A}$ be an element of the minimal degree in $\mathcal{A}$ and let
\[
\mathcal{A}'=\{q-p,\;\partial_{m_j}q-p:\; q\in \mathcal{A},\; 1\leq j\leq d\},
\]
then $\phi(\mathcal{A}')\prec \phi(\mathcal{A})$.
\end{lemma}

\subsection{A stronger result}\

Throughout this section,
let $(X,T)$ and $(Y,T)$ be minimal systems, and
let
\[
X\stackrel{\pi}{\longrightarrow} Y\stackrel{\phi}{\longrightarrow} X/\mathbf{RP}^{[\infty]}(X)=:X_\infty
\]
 be factor maps such that
$\pi$ is \textbf{open} and $\phi$ is \textbf{almost one to one}.

For systems $\mathcal{A}=\{p_1,\ldots,p_s\}$ and $\mathcal{C}$,
we just say that $\pi$ has the property $\Lambda(\mathcal{A},\mathcal{C})$ for convenience,
if for any open subsets $V_0,V_1,\ldots,V_s$ of $X$ with
$\bigcap_{i=0}^s\pi(V_i)\neq \emptyset$,
there exist $z\in V_0$ and $n\in \mathbb{N}$
such that
\begin{enumerate}[itemsep=4pt,parsep=2pt,label=(\arabic*)]
\item $T^{p_i(n)}z\in V_i$ for $1\leq i\leq s$;
\item\label{AAAA1111} $T^{q(n)}\pi(z)\in \bigcap_{i=0}^s \pi(V_i)$ for $q\in \mathcal{C}$.
\end{enumerate}

It follows from Theorem \ref{key-thm0}
that to show Theorem \ref{polynomial-TCF},
it suffices to show the following stronger result:

\begin{theorem}\label{polynomial-case}
For any systems $\mathcal{A}$ and $\mathcal{C}$,
$\pi$ has the property $\Lambda(\mathcal{A},\mathcal{C})$.
\end{theorem}

\subsubsection{Ideas for the proof of Theorem \ref{polynomial-case}}\

To prove Theorem \ref{polynomial-case}, we will use induction on the weight vector of $\mathcal{A}$.
The first step we do is to show that
$\pi$ has the property $\Lambda(\mathcal{A},\mathcal{C})$
if the weight vector of $\mathcal{A}$ is $(s,1)$, i.e., $\mathcal{A}=\{c_1n,\ldots,c_sn\}$,
where $c_1,\ldots,c_s$ are distinct non-zero integers.
In the second
step we assume that $\pi$ has the property $\Lambda(\mathcal{A}',\mathcal{C}')$
for any system $\mathcal{C}'$ and the system $\mathcal{A}'$ whose weight vector is
$\prec \big( (\phi(w_1),w_1),\ldots,(\phi(w_k),w_k)\big)$.
Then we show that
$\pi$ also has the property $\Lambda(\mathcal{A},\mathcal{C})$
for any system $\mathcal{C}$ and
%the result also holds for
the system $\mathcal{A}$
with weight vector $\big( (\phi(w_1),w_1),\ldots,(\phi(w_k),w_k)\big)$,
and hence the proof is completed.

Before giving the proof of the second step,
we show Theorem \ref{polynomial-case} holds for system $\mathcal{A}$ with weight vectors $(1,2)$ and $\big((s,1),(1,2)\big)$
as examples
to illustrate our basic ideas.

\subsubsection{The concrete construction}\

We use a simple example to describe how we prove Theorem \ref{polynomial-case}.

Let $(X,T)$ be a minimal system, and assume $\pi:X\to X/\mathbf{RP}^{[\infty]}(X)= X_{\infty}$
is open. For open subsets $U,V$ of $X$
with $\pi(U)\cap \pi(V)\neq\emptyset$, we aim to choose some $n\in \mathbb{Z}$ with
\begin{equation}\label{simple-example}
U\cap T^{-n^2}V\neq \emptyset.
\end{equation}

\noindent {\bf Construction.}
The classical idea (under some assumption) to show (\ref{simple-example}) is the following:\

\noindent {(i):} Cover $X$ by the orbits of $U,V$. That is, choose $d\in \mathbb{N}$
with $\bigcup_{i=1}^d T^iU=X=\bigcup_{i=1}^d T^iV$;\

\noindent {(ii):} Cnstruct $x_1,\ldots,x_d\in X$ and $n_1,\ldots,n_d\in \mathbb{N}$
such that
\[
T^{(n_i+\ldots+n_j)^2}x_j\in T^iV,\quad \forall\; 1\leq i\leq j\leq d.
\]
Once we have achieved this, then $x_d\in T^lU$ for some $ l\in\{1,\ldots, d\}$
which implies
\[
U\cap T^{-(n_l+\ldots+n_d)^2}V\neq \emptyset.
\]

\medskip

For general minimal systems,
this construction needs some modifications.
Now consider the return time set
\begin{equation}\label{return time set-a}
N_k(W_1,W_2):=\{n\in \mathbb{Z}:W_1\cap T^{-kn}W_2\neq \emptyset\}.
\end{equation}
It follows from Theorem \ref{key-thm} that $N_k(W_1,W_2)$ is non-empty for any
non-zero integer $k$ and any open subsets $W_1,W_2$ of $X$ with $\pi(W_1)\cap \pi(W_2)\neq \emptyset$.
%Note that this conclusion may not hold for general minimal systems without any assumption.
Thus to ensure the existence of the construction (i), %for general minimal systems,
the first change we make is the following:

\medskip

\noindent {(i)*:} Cover some fixed fiber instead of the whole space.
That is,
choose some $x\in X$ with $\pi(x)\in \pi(U)\cap \pi(V)$ and choose $a_1,\ldots,a_d\in \mathbb{N}$ such that
\[
\pi^{-1}( \{\pi(x)\})\subset \big(\bigcup_{j=1}^dT^{a_j}U\big)
\cap
\big(\bigcup_{j=1}^dT^{a_j}V\big).
\]

Let us go into the detail  of the first two steps in construction (ii).
Choose $x_1\in X$ and $n_1\in \mathbb{N}$ with $ T^{n_1^2}x_1 \in T^{a_1}V$.
Note that by Bergelson-Leibman Theorem \cite{BL96} such choice exists.
What additional condition the point $x_1$ need to satisfy remains to be determined.

For the choice of $x_2$,
a feasible method is to track
$x_1$ along $\partial_{n_1}n^2=n^2+2n_1n$.
To be more precise, it suffices to choose $x_2\in X$ and $n_2\in \mathbb{N}$
such that $T^{n_2^2}x_2\in T^{a_2}V$ and $T^{n_2^2+2n_1n_2}x_2\in V_{x_1}$,
where $V_{x_1}$ is an open neighbourhood of $x_1$ with $T^{n_1^2}V_{x_1}\subset T^{a_1}V$.
This implies that the return time set $\{n\in \mathbb{Z}:T^{a_2}V\cap T^{-2n_1n}V_{x_1}\neq \emptyset\}$
should be non-empty.
By the argument above about the set (\ref{return time set-a}),
a suitable condition is $\pi(T^{a_2}V)\cap \pi(V_{x_1})\neq \emptyset$.
Thus to guarantee the induction procedure,
the points we choose in construction (ii) need to be very close to the fixed fiber,
i.e., $T^{n_1^2}\pi(x_1),T^{n_2^2}\pi(x_2)\in \pi(U)\cap \pi(V)$:

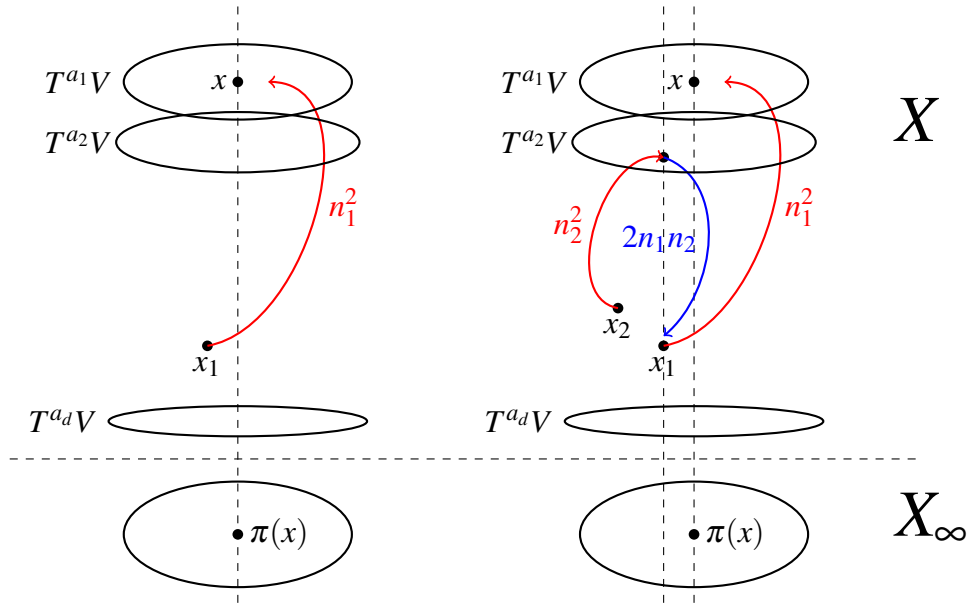
\bigskip

\[
\begin{tikzpicture}
\draw [dashed,-] (-3,4)--(-3,-4);
\draw [dashed,-] (2.6,4)--(2.6,-4);
\draw [dashed,-] (3,4)--(3,-4);
\draw  [dashed,-]  (-6,-2)--(6,-2);
\node [right] at (-3,-3) {$\pi(x)$};
\node [left] at (-4.5,3) {$T^{a_1}V$};
\node [left] at (-4.5,2.2) {$T^{a_2}V$};
\node [left] at (-4.7,-1.5) {$T^{a_d}V$};
\node [left] at (1.5,3) {$T^{a_1}V$};
\node [left] at (1.5,2.2) {$T^{a_2}V$};
\node [left] at (1.3,-1.5) {$T^{a_d}V$};
\node [right] at (3,-3) {$\pi(x)$};
\node [below] at (-3.4,-0.5){$x_1$};
\node [below] at (2.6,-0.5){$x_1$};
\node [below] at (2,0){$x_2$};
\node [left] at (3,3) {$x$};
\node [left] at (-3,3) {$x$};
\fill
(-3,-3)circle (2pt)
(3,-3)circle (2pt)
(3,3)circle (2pt)
(-3.4,-0.5)circle (2pt)
(2.6,-0.5)circle (2pt)
(2.6,2)circle (2pt)
(-3,3)circle (2pt)
(2,0)circle (2pt);
\draw[thick,red,->] (-3.4,-0.5) to[out=10, in=360] node[right, midway] {$n_1^2$} (-2.6,3);
\draw[thick,red,->] (2.6,-0.5) to[out=10, in=360] node[right, midway] {$n_1^2$} (3.4,3);
\draw[thick,blue,->] (2.6,2) to[out=-20, in=40] node[left, midway] {$2n_1n_2$} (2.6,-0.37);
\draw[thick,red,->] (2,0) to[out=170, in=170] node[left, midway] {$n_2^2$} (2.6,2);
\draw[ thick] (-3, -3) ellipse (1.5 and 0.7);
\draw[thick] (3, -3) ellipse (1.5 and 0.7);
\draw[ thick] (3, 3) ellipse (1.5 and 0.5);
\draw[ thick] (3, 2.2) ellipse (1.6 and 0.4);
\draw[thick] (3, -1.5) ellipse (1.7 and 0.2);
\draw[ thick] (-3, 3) ellipse (1.5 and 0.5);
\draw[ thick] (-3, 2.2) ellipse (1.6 and 0.4);
\draw[ thick] (-3, -1.5) ellipse (1.7 and 0.2);
\node [right,font=\Huge] at (5.5,-2.8) {$X_\infty$};
\node [right,font=\Huge] at (5.5,2.5) {$X$};
\end{tikzpicture}
\]

\medskip

Furthermore,
 we will reduce the system of any given complexity to the lower one,
%for the completeness of the induction procedure,
thus the points chosen in construction (ii) also need to be close to the fixed fiber
along polynomials of any higher degrees,
that is why we need property \ref{AAAA1111} in Theorem \ref{polynomial-case} additionally.

Summarize it as follows:

\medskip

\noindent {(ii)*:}
For any system $\mathcal{C}$, construct $x_1,\ldots,x_d\in X$ and $n_1,\ldots,n_d\in \mathbb{N}$
such that
\begin{enumerate}
[itemsep=4pt,parsep=2pt,label=(\arabic*)]
\item $T^{(n_i+\ldots+n_j)^2}x_j\in T^iV$ for $1\leq i\leq j\leq d$;
\item$T^{q(n)}\pi(x_j)\in \pi(U)\cap \pi(V)$ for $1\leq j\leq d$ and $q\in \mathcal{C}$.% where $\mathcal{C}$ is a system.
\end{enumerate}

\medskip

Practically, we will use construction (i)* and (ii)* for general cases to prove Theorem \ref{polynomial-case}.
When doing this, we find that if in the collection of polynomials there are linear
elements and other non-linear elements, the argument will be very involved.
We will explain in Subsection \ref{exampless} how to overcome this difficulty via proving Case \ref{case2}.

\subsection{The first step: $\mathcal{A}=\{c_1n,\ldots,c_sn\}$}

\begin{lemma}\label{linear-case-with-constraint}
If there are distinct non-zero integers $c_1,\ldots,c_s$ such that $\mathcal{A}=\{c_1n,\ldots,c_sn\}$,
then $\pi$ has the property $\Lambda(\mathcal{A},\mathcal{C})$ for
any system $\mathcal{C}$.
\end{lemma}

\begin{proof}
Fix distinct non-zero integers $c_1,\ldots,c_s$.
Let $V_0,V_1,\ldots,V_s$ be open subsets of $X$ with
$W:=\bigcap_{i=0}^s\pi(V_i)\neq \emptyset$.
Recall that the map $\pi$ is open, thus $W$ is an open subset of $Y$
and $\pi^{-1}( W)\cap V_i\neq \emptyset$
for every $0\leq i\leq s$.

\medskip

\noindent {\bf Case 1:}  $\min\{ c_1,\ldots,c_s\}>0$.

Fix a system $\mathcal{C}$.
By Corollary \ref{return-time-AA},
there is a dense $G_\delta$ subset $\Omega_Y$ of $Y$
such that for any $y\in \Omega_Y$,
$\bigcap_{q\in \mathcal{C}}N_q(y,V)\in \mathcal{F}_{fip}^*$
for every open neighbourhood $V$ of $y$.

By Lemma \ref{fip-infi-fiber}, there exists
a dense $G_\delta$ subset $\Omega_X$ of $X$ such that for any $x \in\Omega_X \cap V_0 \cap \pi^{-1}(W)$,
there exists $A\in \mathcal{F}_{fip}$
such that $T^{c_in}x\in V_i$ for $1\leq i\leq s$ and $n\in A$.

Let $z\in \pi^{-1}(\Omega_Y)\cap \Omega_X  \cap V_0 \cap \pi^{-1}(W)$.
%Without loss of generality, assume that $z\in \pi^{-1}(\Omega_Y)$.
Then we have
\[
\{n\in \mathbb{Z}: \pi(z)\in \bigcap_{q\in \mathcal{C}}T^{-q(n)}W\}=
\bigcap_{q\in \mathcal{C}}N_q(\pi(z),W)\in \mathcal{F}_{fip}^*,
\]
and we can choose some $n\in \mathbb{Z}$ such that% $T^{c_in}z\in V_i$ for $1\leq i\leq s$
%and $T^{q(n)}\pi(z)\in W=\bigcap_{i=0}^s\pi(V_i)$ for $q\in \mathcal{C}$.
\begin{enumerate}[itemsep=4pt,parsep=2pt,label=(\arabic*)]
\item $T^{c_in}z\in V_i$ for $1\leq i\leq s$;
\item\label{AAAA1111} $T^{q(n)}\pi(z)\in W=\bigcap_{i=0}^s \pi(V_i)$ for $q\in \mathcal{C}$.
\end{enumerate}

Thus $\pi$ has the property $\Lambda(\mathcal{A},\mathcal{C})$.

\medskip

\noindent {\bf Case 2:} $\min\{ c_1,\ldots,c_s\}<0$.

Fix a system $\mathcal{C}$.
Assume $c_m=\min\{ c_1,\ldots,c_s\}<0$ for some $m\in \{1,\ldots,s\}$.
Let
\begin{align*}
\mathcal{A}' &=\{-c_mn,\; (c_1-c_m)n,\ldots,(c_{m-1}-c_m)n,\; (c_{m+1}-c_m)n,\ldots,(c_s-c_m)n\}, \\
\mathcal{C}'& =\{q(n)-c_mn:q\in \mathcal{C}\}.
\end{align*}

By Case 1, $\pi $ has the property $\Lambda(\mathcal{A}',\mathcal{C}')$.
Then for open sets $V_m,V_0,\ldots,V_{m-1},V_{m+1},\ldots,V_s$,
there exist $w\in V_m$ and $n\in \mathbb{Z}$ such that
$T^{-c_mn}w\in V_0,T^{(c_i-c_m)n}w\in V_i$ for $i\in \{1,\ldots,s\}\backslash\{m\}$
and $T^{q(n)-c_mn}\pi(w)\in \bigcap_{i=0}^s\pi(V_i)$ for $q\in \mathcal{C}$.
By letting $z=T^{-c_mn}w$, we deduce that $\pi$ has the property $\Lambda(\mathcal{A},\mathcal{C})$.

This completes the proof.
\end{proof}

\subsection{\bf Examples}\label{exampless}\

In this subsection,
we show Theorem \ref{polynomial-case} holds for system $\mathcal{A}$ with weight vectors $(1,2)$
and $\big((s,1),(1,2)\big)$.

\begin{case}\label{case1}
$\phi(\mathcal{A})=(1,2)$.
\end{case}

\begin{proof}
Let $\mathcal{A}=\{an^2+b_1n,\ldots,an^2+b_tn  \}$, where $a$ is a non-zero integer
and $b_1,\ldots,b_t$ are distinct integers.
Let $V_0,V_1,\ldots,V_m$ be open subsets of $X$ with
$W:=  \bigcap_{m=0}^t\pi(V_m)\neq \emptyset.$
Recall that the map $\pi$ is open, thus $W$ is an open subset of $Y$.
For $0\leq m\leq t$,
by replacing $V_m$ by $V_m\cap  \pi^{-1}(W)$
respectively,
we may assume without loss of generality that $\pi(V_m)=W$.

As $(X,T)$ is minimal, there is some $N\in \mathbb{N}$ such that
$\bigcup_{j=1}^NT^j V_m=X$ for every $0\leq m\leq t$.
Let $x\in X$ with $\pi(x)\in W$ and let
\[
\{1\leq j\leq N:\pi(x)\in T^jW\}
=\{a_1,\ldots,a_d\}.
\]
Then we have
\[
\pi^{-1}( \{\pi(x)\})\subset \bigcap_{m=0}^t\big(\bigcup_{j=1}^dT^{a_j}V_m\big) .
\]
As the map $\pi$ is open,
by Theorem \ref{open-map} we can choose $\delta>0$ such that
\begin{equation}\label{relation-1}
\pi^{-1}\big(B(\pi(x),\delta)\big)\subset \bigcap_{m=0}^t \big(\bigcup_{j=1}^dT^{a_j}V_m\big),
\end{equation}
and
\begin{equation}\label{relation-2}
B(\pi(x),\delta)\subset W \cap \big(\bigcap_{j=1}^d  T^{a_j}W \big).
\end{equation}

Fix a system $\mathcal{C}$.
Write $p_m(n)=an^2+b_mn$ for $1\leq m\leq t$ and let $\eta=\delta/d$.
Inductively we will construct $x_1,\ldots,x_d\in X$
and $n_1,\ldots,n_d\in \mathbb{N}$
such that for $1\leq j\leq l \leq d$,
\begin{itemize}[itemsep=4pt,parsep=2pt]
\item $\pi(x_l)\in  B(\pi(x),l\eta)$;
\item $T^{p_m(n_j+\ldots+n_l)}x_l\in T^{a_j}V_m$ for $1\leq m\leq t$;
\item $T^{q(n_j+\ldots+n_{l})}\pi(x_l)\in  B(\pi(x),l\eta)$ for $q\in \mathcal{C}$.
\end{itemize}

Assume this has been achieved, then for $1\leq j \leq d$,
\begin{itemize}[itemsep=4pt,parsep=2pt]
\item $\pi(x_d)\in B(\pi(x),d\eta)=B(\pi(x),\delta)$;
\item$T^{p_m(n_j+\ldots+n_d)}x_d\in T^{a_j}V_m$ for $1\leq m\leq t$;
\item $T^{q(n_j+\ldots+n_d)}\pi(x_d)\in  B(\pi(x),d\eta)=
B(\pi(x),\delta)\subset \bigcap_{j=1}^d  T^{a_j}W $ for $q\in \mathcal{C}$ by (\ref{relation-2}).
\end{itemize}

As $\pi(x_d)\in B(\pi(x),\delta)$,
it follows from (\ref{relation-1}) that $x_d\in T^{a_j}V_0$ for some $ j\in\{1,\ldots, d\}$.
Put $z=T^{-a_j}x_d$ and $n=n_j+\ldots+n_d$, then we have
\begin{itemize}[itemsep=4pt,parsep=2pt]
\item $z=T^{-a_j}x_d\in V_0$;
\item $T^{p_m(n)}z=T^{-a_j}(T^{p_m(n)}x_d)\in V_m$ for $1\leq m\leq t$;
\item $T^{q(n)}\pi(z)=T^{-a_j}(T^{q(n)}\pi(x_d))\in W =  \bigcap_{m=0}^t\pi(V_m)$ for $q\in \mathcal{C}$.
\end{itemize}

This shows that $\pi$ has the property $\Lambda(\mathcal{A},\mathcal{C})$.

\medskip

{\bf We now return to the inductive construction of $x_1,\ldots,x_d$ and $n_1,\ldots,n_d$.}

\medskip

\noindent {\bf Step 1:}
Let $I_1=\pi^{-1}\big(B(\pi(x),\eta)\big)$. Then $I_1$ is an open subset of $X$ and
\[
\pi(x)\in \underbrace{\bigcap_{m=1}^t\pi( I_1\cap T^{a_1} V_m) }_{=:S_1}\subset  B(\pi(x),\eta).
\]

\noindent {(i)} When $t=1$, i.e., $\mathcal{A}=\{p_1\}$.
By Bergelson-Leibman Theorem \cite{BL96}, there exsit $x_1\in I_1\cap T^{a_1}V_1$
and $n_1\in \mathbb{N}$ such that for $q\in \mathcal{C}$,
\[
T^{p_1(n_1)}x_1,\; T^{q(n_1)}x_1\in I_1\cap T^{a_1}V_1.
\]
Then for $q\in \mathcal{C}$ we have
\[
\pi(x_{1}),\; T^{q(n_{1})}\pi(x_{1})\in  B(\pi(x),\eta).
\]
 \noindent {(ii)} When $t\geq 2$.
Let
\begin{align*}
\mathcal{A}_1 &=\{ p_m-p_1:2\leq m\leq t\}=\{   (b_m-b_1) n:2\leq m \leq t\}, \\
\mathcal{C}_1 & =\{-p_1,\; q-p_1:q\in \mathcal{C}\}.
\end{align*}

By Lemma \ref{linear-case-with-constraint}, $\pi$ has the property $\Lambda(\mathcal{A}_1,\mathcal{C}_1)$.
Then for open sets $I_1\cap T^{a_1}V_1$ and $I_1\cap T^{a_1}V_2,\ldots,I_1\cap T^{a_1}V_t$,
there exist $y_1\in I_1\cap T^{a_1}V_1$ and $n_1\in \mathbb{N}$
such that
\begin{enumerate}[itemsep=4pt,parsep=2pt,label=(\arabic*)]
\item[$(1a)$] $T^{p_m(n_1)-p_1(n_1)}y_1\in I_1\cap T^{a_1} V_m$ for $2\leq m\leq t$;
\item[$(1c)$] $T^{-p_1( n_1)}\pi(y_1),\;T^{q(n_1)-p_1( n_1)}\pi(y_1)\in S_1$ for $q\in \mathcal{C}$.
\end{enumerate}

Set $x_1=T^{-p_1( n_1)}y_1$.
By $(1a)$, for $1\leq m\leq t$ we have
\[
T^{p_m(n_1)}x_1\in T^{a_1}V_m.
\]
By $(1c)$, for $ q\in \mathcal{C}$ we have
\[
\pi(x_{1}),\; T^{q(n_{1})}\pi(x_{1})\in S_1\subset B(\pi(x),\eta).
\]

Thus by (i) and (ii) we can choose $x_1\in X$ with $\pi(x_1)\in B(\pi(x),\eta)$ and $n_1\in \mathbb{N}$
such that
\begin{itemize}[itemsep=4pt,parsep=2pt]
\item $T^{p_m(n_{1})}x_{1}\in T^{a_1}V_m$ for $1\leq m\leq t$ ;
\item $T^{q(n_{1})}\pi(x_{1})\in  B(\pi(x),\eta)$ for $q\in \mathcal{C}$.
\end{itemize}

\medskip

\noindent {\bf Step l:}
Let $l\geq 2$ be an integer and assume that we have already chosen
$x_1,\ldots,x_{l-1}\in X$ and $n_1,\ldots,n_{l-1}\in \mathbb{N}$ such that for $1\leq j \leq l-1$,
\begin{itemize}[itemsep=4pt,parsep=2pt]
\item $ \pi(x_{l-1})\in  B(\pi(x),(l-1)\eta)\subset B(\pi(x),\delta)$;
\item $T^{p_m(n_j+\ldots+n_{l-1})}x_{l-1}\in  T^{a_j}V_m$ for $1\leq m \leq t$;
\item $  T^{q(n_j+\ldots+n_{l-1})}\pi(x_{l-1})\in  B(\pi(x),(l-1)\eta)$ for $q\in \mathcal{C}$.
\end{itemize}

Choose $\eta_l>0$ with $\eta_l<\eta$ such that for $1\leq j \leq l-1$,
\begin{align}
\label{newl1}  T^{p_m(n_j+\ldots+n_{l-1})}B(x_{l-1},\eta_l)&\subset T^{a_j}V_m,
&\forall\;  1\leq m \leq t, \\
\label{newl2}   T^{q(n_j+\ldots+n_{l-1})}B\big(\pi(x_{l-1}),\eta_l\big)&
\subset B(\pi(x),(l-1)\eta),&\forall \; q\in \mathcal{C}.
\end{align}

Let $I_l=\pi^{-1}\big(B(\pi(x_{l-1}),\eta_l)\big)$.
By (\ref{relation-2}),
we have $ \pi(x_{l-1})\in B(\pi(x),\delta)\subset \bigcap_{j=1}^dT^{a_j}W$
 and
\[
\pi(x_{l-1})\in \underbrace{\bigcap_{m=1}^t\pi( I_l\cap T^{a_l} V_m)\cap
\pi\big(B(x_{l-1},\eta_l)\big) }_{=:S_l}\subset \pi(I_l)=
B(\pi(x_{l-1}),\eta_l)\subset B(\pi(x_{l-1}),\eta).
\]

Let
\begin{align*}
\mathcal{A}_l= &\{ p_m-p_1,\;\partial_{n_j+\ldots+n_{l-1}}p_m-p_1:1\leq m\leq t,1\leq j\leq l-1\} \\
= & \{  (b_m-b_1) n,\; (b_m-b_1+2an_j+\ldots+2an_{l-1}) n:1\leq m\leq t,1\leq j\leq l-1\},\\
\mathcal{C}_l=&\{-p_1,\; q-p_1,\;\partial_{n_j+\ldots+n_{l-1}}q-p_1:q\in \mathcal{C},1\leq j\leq l-1\}.
\end{align*}

By Lemma \ref{linear-case-with-constraint}, $\pi$ has the property $\Lambda(\mathcal{A}_l,\mathcal{C}_l)$.
Then for open sets 
$I_l\cap T^{a_l}V_1,I_l\cap T^{a_l}V_2, \ldots,I_l\cap T^{a_l}V_t ,
\underbrace{B(x_{l-1},\eta_l),\ldots,B(x_{l-1},\eta_l)}_{t(l-1)\; \mathrm{times}}$,
there exist $y_l\in I_l\cap T^{a_l}V_1$ and $n_l\in \mathbb{N}$
such that for $1\leq j \leq l-1$,
\begin{enumerate}[itemsep=4pt,parsep=2pt,label=(\arabic*)]
\item[$(la)_1$]$ T^{p_m(n_l)-p_1(n_l)}y_l\in  I_l\cap T^{a_l}V_m$ for $2\leq m\leq t$;
\item[$(la)_2$] $ T^{\partial_{n_j+\ldots+n_{l-1}}p_m(n_l)-p_1(n_l)}y_l\in B(x_{l-1},\eta_l)$
for $1\leq m\leq t$;
\item[$(lc)_1$] $ T^{-p_1(n_l)}\pi(y_l),\;T^{q(n_l)-p_1(n_l)}\pi(y_l)\in  S_l$ for $q\in \mathcal{C}$;
\item[$(lc)_2$] $ T^{\partial_{n_j+\ldots+n_{l-1}}q(n_l)-p_1(n_l)}\pi(y_l)\in  S_l
        \subset B(\pi(x_{l-1}),\eta_l)$ for $q\in \mathcal{C}$.
\end{enumerate}

Set $x_l=T^{-p_1( n_l)}y_l$.
By $(la)_1$ for $1\leq m\leq t$ we have
\[
T^{p_m(n_l)}x_l\in T^{a_l}V_m.
\]
By $(la)_2$ and (\ref{newl1}), for $1\leq m\leq t,1\leq j\leq l-1$ we have
\begin{align*}
T^{p_m(n_j+\ldots+n_{l-1}+n_l)}x_l=& T^{p_m(n_j+\ldots+n_{l-1})}(T^{\partial_{n_j+\ldots+n_{l-1}}p_m(n_l)}x_l)\ \\
   & \in T^{p_m(n_j+\ldots+n_{l-1})} B(x_{l-1},\eta_l)\subset T^{a_j}V_m.
\end{align*}
By $(lc)_1$, for $q\in \mathcal{C}$ we have
\[
\pi(x_l), \;T^{q(n_l)}\pi(x_l)\in S_l\subset B(\pi(x_{l-1}),\eta)
\subset B(\pi(x),l\eta).
\]
By $(lc)_2$ and (\ref{newl2}), for $q\in \mathcal{C},1\leq j\leq l-1$ we have
\begin{align*}
T^{q(n_j+\ldots+n_{l-1}+n_l)}\pi(x_l)=&T^{q(n_j+\ldots+n_{l-1})}(T^{\partial_{n_j+\ldots+n_{l-1}}q(n_l)}\pi(x_l))\ \\
   & \in T^{q(n_j+\ldots+n_{l-1})} B(\pi(x_{l-1}),\eta_l)\subset B(\pi(x),l\eta).
\end{align*}

We finish the construction by induction.
\end{proof}

\medskip

\begin{case}\label{case2}
 $\phi(\mathcal{A})=\big((s,1),(1,2)\big)$.
\end{case}

From $(la)_1,(la)_2$  in the proof of Case \ref{case1}, we can see that the points
$T^{p_m(n_l)}x_l$ and $T^{\partial_{n_j+\ldots+n_{l-1}}p_m(n_l)}x_l$
should be in different open sets.
However, this may not hold for the family of polynomials containing linear ones.
The reason is clear,
for any linear polynomial $q\in \mathcal{P}^*$,
one has $\partial_mq=q$ for all $m\in \mathbb{Z}$.
In this case, we cannot use $q$ and $\partial_mq$ to track different open sets.

%as the difference of  with zero constant is nothing but itself.

%in the construction $x_l$ and $n_l$ should satisfy $T$
%of the $l$-th point,
%we track the $(l-1)$-th point along the differences of the polynomials $p_1,\ldots,p_t$
%to control the positions of the first $(l-1)$ times,
%and at the same time control the $l$-th position along the polynomials $p_1,\ldots,p_t$.
To overcome this difficulty, we divide the proof of Case \ref{case2} into
the following two claims whose proofs will be given after the proof of Case \ref{case2} since they are very long.
For general cases, the idea is similar.

\begin{claim}\label{ex2}
Let $\mathcal{A}=\{an^2+b_1n,\ldots,an^2+b_tn  \}$, where $a$ is a non-zero integer
and $b_1,\ldots,b_t$ are distinct integers,
and let $c_1,\ldots,c_s$ be distinct non-zero integers.
Then for any system $\mathcal{C}$ and open subsets $V_0,V_1,\ldots,V_m$ of $X$
with $\bigcap_{m=0}^t\pi(V_m)\neq \emptyset$,
there exist $z\in V_0$ and $n\in \mathbb{N}$
such that
\begin{itemize}[itemsep=4pt,parsep=2pt]
\item$ T^{c_i n}z\in V_0$ for $1\leq i\leq s$;
\item $ T^{an^2+b_mn}z\in V_m$ for $1\leq m\leq t$;
\item $T^{q(n)}\pi(z)\in \bigcap_{m=0}^t\pi(V_m)$ for $q\in \mathcal{C}$.
\end{itemize}
\end{claim}

\begin{claim}\label{ex-claim2}
Let $\mathcal{A}=\{an^2+b_1n,\ldots,an^2+b_tn  \}$, where $a$ is a non-zero integer
and $b_1,\ldots,b_t$ are distinct integers,
and let $c_1,\ldots,c_s$ be distinct non-zero integers.
Then for any system $\mathcal{C}$ and open subsets $W_0,W_1,\ldots,W_s$ of $X$
with $\bigcap_{i=0}^s\pi(W_i)\neq \emptyset$,
there exist $z\in W_0$ and $n\in \mathbb{N}$
such that
\begin{itemize}[itemsep=4pt,parsep=2pt]
\item $ T^{an^2+b_mn}z\in W_0$ for $1\leq m\leq t$;
\item$ T^{c_in}z\in W_i$ for $1\leq i \leq s$;
\item $T^{q(n)}\pi(z)\in \bigcap_{i=0}^s\pi(W_i)$ for $q\in \mathcal{C}$.
\end{itemize}
\end{claim}

Using Claims \ref{ex2} and \ref{ex-claim2}, we are able to give a proof of Case \ref{case2}.

\begin{proof}[Proof of Case \ref{case2} assuming Claims \ref{ex2} and \ref{ex-claim2}]
Fix a system $\mathcal{C}$.
Let
\[
\mathcal{A}=\{an^2+b_1n,\ldots,an^2+b_tn,\; c_1n,\ldots,c_s n \}
\]
 where $a$ is a non-zero integer,
$b_1,\ldots,b_t$ are distinct integers, and
$c_1,\ldots,c_s$ are distinct non-zero integers.
Let $V_0,V_1,\ldots,V_s,U_1,\ldots,U_t$ be open subsets of $X$ with
\[
W:=\bigcap_{i=0}^s\pi(V_i)\cap\bigcap_{m=1}^t\pi(U_m)\neq \emptyset.
\]

Let $\mathcal{A}_1=\{an^2+b_1n,\ldots,an^2+b_tn\}$.
Using Claim \ref{ex2} for system $\mathcal{A}_1$ and integers $c_1,\ldots,c_s$,
then for system $\mathcal{C}$ and open sets $V_0\cap \pi^{-1}(W),U_1,\ldots,U_t$,
there exist $w\in V_0\cap \pi^{-1}(W)$ and $k\in \mathbb{N}$
such that
\begin{itemize}[itemsep=4pt,parsep=2pt]
\item $ T^{c_ik}w\in V_0\cap \pi^{-1}(W)$ for $1\leq i \leq s$;
\item $ T^{ak^2+b_mk}w\in U_m$ for $1\leq m\leq t$;
\item$T^{q(k)}\pi(w)\in \pi\big(V_0\cap \pi^{-1}(W)\big)\cap \bigcap_{m=1}^t\pi(U_m)\subset W$ for $q\in \mathcal{C}$.
\end{itemize}

For every $1\leq i \leq s$,
as $T^{c_ik}w\in V_0\cap \pi^{-1}(W)$,
there is some $w_i\in V_i$ with $\pi(T^{c_ik}w)=\pi(w_i)$ which implies
\begin{equation}\label{equal1}
\pi(w)=\pi(T^{-c_ik}w_i).
\end{equation}

Choose $\gamma>0$ such that
\begin{align}
\label{conA1} T^{c_ik}B(T^{-c_ik}w_i,\gamma)&\subset V_i, \quad\;\;\; \forall\; 1\leq i\leq s, \\
\label{conA2} T^{ak^2+b_mk}B(w,\gamma)&\subset U_m, \quad\; \forall\; 1\leq m\leq t,\\
\label{conA3} T^{q(k)}B\big(\pi(w),\gamma\big)&\subset W,\quad\;\; \;\forall\; q\in \mathcal{C}.
\end{align}

Let $W_0=B(w,\gamma)\cap \pi^{-1}\big( B(\pi(w),\gamma)\big)\cap V_0$ and
let $W_i=B(T^{-c_ik}w_i,\gamma)$ for $1\leq i\leq s$.
Then $W_0$ is a non-empty open set as $w\in W_0$,
and $\pi(w)\in \pi(W_i)$ by (\ref{equal1}).

Let
\begin{align*}
\mathcal{A}_2 &=\{\partial_k p:p\in \mathcal{A}_1\}=\{an^2+(b_1+2ak)n,\ldots,an^2+(b_t+2ak)n\}, \\
\mathcal{C}_1 & =\{\partial_k q:q\in \mathcal{C}\}.
\end{align*}

Now using Claim \ref{ex-claim2} for system $\mathcal{A}_2$ and integers $c_1,\ldots,c_s$,
then for system $\mathcal{C}_1$ and open sets $W_0,W_1,\ldots,W_s$,
there exist $z\in W_0$ and $l\in \mathbb{N}$ such that
\begin{align}
\label{qqq1} T^{al^2+(b_m+2ak)l}z&\in W_0\subset B(w,\gamma), \quad\quad\quad\quad\quad\quad\quad\quad\;\;\;\; \forall\; 1\leq m\leq t, \\
\label{qqq2} T^{c_il}z&\in W_i=B(T^{-c_ik}w_i,\gamma), \quad\quad\quad\quad\quad\quad\;\; \;\forall\; 1\leq i\leq s,\\
\label{qqq3}T^{\partial_kq(l)}\pi(z)&\in \bigcap_{i=0}^s\pi(W_i)\subset \pi(W_0)\subset B(\pi(w),\gamma),\quad\;\; \forall\; q\in \mathcal{C}.
\end{align}
By (\ref{conA1}) and (\ref{qqq2}), for $1\leq i \leq s$ we have
\[
T^{c_i(l+k)}z\in T^{c_ik}B(T^{-c_ik}w_i,\gamma)\subset V_i.
\]
By (\ref{conA2}) and (\ref{qqq1}), for $1\leq m\leq t$ we have
\[
T^{a(l+k)^2+b_m(l+k)}z=T^{ak^2+b_mk}(T^{al^2+(b_m+2ak)l}z)\in T^{ak^2+b_mk}B(w,\gamma)\subset U_m.
\]
By (\ref{conA3}) and (\ref{qqq3}), for $q\in \mathcal{C}$ we have
\[
T^{q(l+k)}\pi(z)=T^{q(k)}(T^{\partial_k q(l)}\pi(z))\in T^{q(k)}B(\pi(w),\gamma)\subset W.
\]
Put $n=l+k$, then we have
\begin{itemize}[itemsep=4pt,parsep=2pt]
\item $z\in V_0$;
\item$ T^{c_in}z\in V_i$ for $1\leq i \leq s$;
\item $ T^{an^2+b_mn}z\in U_m$ for $1\leq m\leq t$;
\item $T^{q(n)}\pi(z)\in W=\bigcap_{i=0}^s\pi(V_i)\cap\bigcap_{m=1}^t\pi(U_m)$ for $q\in \mathcal{C}$.
\end{itemize}

This completes the proof of Case \ref{case2}.
\end{proof}

We now proceed to the proof of Claims \ref{ex2} and \ref{ex-claim2}.

\begin{proof}[Proof of Claim \ref{ex2}]
%Let $a$ be a non-zero integer and let $b_1,\ldots,b_t$ be distinct integers.
We show this claim by induction on $s$.
When $s=0$, it follows from Case \ref{case1}.
Let $s\geq 1$ be an integer and
suppose the statement of the claim is true for $ (s-1)$.

Let $c_1,\ldots,c_s$ be distinct non-zero integers,
and let $V_0,V_1,\ldots,V_m$ be open subsets of $X$ with $W:=\bigcap_{m=0}^t\pi(V_m)\neq \emptyset$.
By the similar argument in the proof of Case \ref{case1},
we may assume without loss of generality that $\pi(V_m)=W$ for $0\leq m\leq t$,
and there exist $x\in X$ with $\pi(x)\in W$, $a_1,\ldots,a_d\in \mathbb{N}$ and $\delta>0$ such that
\begin{equation}\label{case2-relation1}
\pi^{-1}\big(B(\pi(x),\delta)\big)\subset  \bigcap_{m=0}^t\bigcup_{j=1}^dT^{a_j}V_m,
\end{equation}
and
\begin{equation}\label{case2-relation2}
B(\pi(x),\delta)\subset W \cap \big(\bigcap_{j=1}^d  T^{a_j}W \big).
\end{equation}

Fix a system $\mathcal{C}$ and let $\eta=\delta/d$.
Write $p_m(n)=an^2+b_mn$ for $1\leq m\leq t$.
Inductively we will construct $x_1,\ldots,x_d\in X,k_1,\ldots,k_{d+1}\in \{1,\ldots,d\}$ with $k_1=1$
and $n_1,\ldots,n_d\in \mathbb{N}$
such that for every $1\leq j\leq l \leq d$,
\begin{itemize}[itemsep=4pt,parsep=2pt]
\item $x_l\in T^{a_{k_{l+1}}}V_0$;
\item $\pi(x_l)\in B(\pi(x),l\eta)$;
\item $T^{c_i(n_j+\ldots+n_l)}x_l\in T^{a_{k_j}}V_0$ for $1\leq i\leq s$;
\item $T^{p_m(n_j+\ldots+n_l)}x_l\in  T^{a_{k_j}}V_m$ for $1\leq m\leq t$;
\item $T^{q(n_j+\ldots+n_l)}\pi(x_l)\in  B(\pi(x),l\eta)$ for $q\in \mathcal{C}$.
\end{itemize}

Assume this has been achieved, there exist $1\leq j\leq l\leq d$ with $k_j=k_{l+1}$ such that
\begin{itemize}[itemsep=4pt,parsep=2pt]
\item $x_l\in T^{a_{k_{l+1}}}V_0=T^{a_{k_j}}V_0$;
\item $T^{c_i(n_j+\ldots+n_l)}x_l\in T^{a_{k_j}}V_0$ for $1\leq i\leq s$;
\item $T^{p_m(n_j+\ldots+n_l)}x_l\in T^{a_{k_j}}V_m$ for $1\leq m\leq t$;
\item $T^{q(n_j+\ldots+n_l)}\pi(x_l)\in  B(\pi(x),l\eta)\subset B(\pi(x),\delta)\subset \bigcap_{j=1}^d  T^{a_j}W  $ for $q\in \mathcal{C}$
       \ \ \   by (\ref{case2-relation2}).
\end{itemize}

Put $n=n_j+\ldots+n_l$ and $z=T^{-a_{k_{j}}}x_{l}$, then the claim follows.

\medskip

{\bf
We now return to the inductive construction of $x_1,\ldots,x_d,k_1,\ldots,k_{d+1}$ and $n_1,\ldots,n_d$.
}

\medskip

\noindent {\bf Step 1:}
Let $I_1=\pi^{-1}\big(B(\pi(x),\eta)\big)$
Then $I_1$ is an open subset of $X$ and
\[
\pi(x)\in \underbrace{\bigcap_{m=0}^t\pi( I_1\cap T^{a_1} V_m) }_{=:S_1}\subset \pi(I_1)=
B(\pi(x),\eta).
\]

Let
\[
\mathcal{A}_1 =\{p_m(n)-c_1n:1\leq m\leq t\} \quad\mathrm{and}\quad
\mathcal{C}_1  =\{-c_1n,\; q(n)-c_1n:q\in \mathcal{C}\}.
\]

By our inductive hypothesis,
the conclusion of Claim \ref{ex2} holds for system $\mathcal{A}_1$ and integers $c_2-c_1,\ldots,c_s-c_1$.
Then for system $\mathcal{C}_1$ and open sets $I_1\cap T^{a_1} V_0,I_1\cap T^{a_1} V_1, \ldots,I_1\cap T^{a_1} V_t$,
there exist $y_1\in I_1\cap T^{a_1}V_0$ and $n_1\in \mathbb{N}$
such that
\begin{enumerate}[itemsep=4pt,parsep=2pt,label=(\arabic*)]
\item[$(1a)_1$] $T^{(c_i-c_1) n_1}y_1\in I_1\cap T^{a_1} V_0$ for $2\leq i\leq s$;
\item[$(1a)_2$] $T^{p_m(n_1)-c_1n_1}y_1\in I_1\cap T^{a_1} V_m$ for $1\leq m\leq t$;
\item[$(1c)_{\;\;}$] $T^{-c_1 n_1}\pi(y_1),\; T^{q(n_1)-c_1 n_1}\pi(y_1)\in S_1$ for $q\in \mathcal{C}$.
\end{enumerate}

Set $x_1=T^{-c_1 n_1}y_1$.
By $(1a)_1$, for $1\leq i\leq s$ we have
\[
T^{c_i n_1}x_1\in T^{a_1}V_0.
\]
By $(1a)_2$, for $1\leq m\leq t$ we have
\[
T^{p_m(n_1)}x_{1}\in T^{a_1}V_m.
\]
By $(1c)$, for $q\in \mathcal{C}$ we have
\begin{equation}\label{case2-001}
\pi(x_{1}),\;T^{q(n_{1})}\pi(x_{1})\in S_1\subset  B(\pi(x),\eta)\subset  B(\pi(x),\delta).
\end{equation}

There is some $k_2\in \{1,\ldots,d\}$ with $x_1\in T^{a_{k_2}}V_0$ by (\ref{case2-relation1}) and (\ref{case2-001}).

\medskip

\noindent {\bf Step l:}
Let $l\geq2$ be an integer and assume that we have already chosen
$x_1,\ldots,x_{l-1}\in X$, $k_1,\ldots,k_{l-1}\in \{1,\ldots,d\}$ and $n_1,\ldots,n_{l-1}\in \mathbb{N}$ such that
for $1\leq j\leq l-1$,
\begin{itemize}[itemsep=4pt,parsep=2pt]
\item $\pi(x_{l-1})\in B(\pi(x),(l-1)\eta)$;
\item $T^{c_i(n_j+\ldots+n_{l-1})}x_{l-1}\in T^{a_{k_{j}}}V_0$ for $1\leq i\leq s$;
\item $T^{p_m(n_j+\ldots+n_{l-1})}x_{l-1}\in T^{a_{k_{j}}}V_m$ for $1\leq m\leq t$;
\item $T^{q(n_j+\ldots+n_{l-1})}\pi(x_{l-1})\in  B(\pi(x),(l-1)\eta)$ for $q\in \mathcal{C}$.
\end{itemize}

As $\pi(x_{l-1})\in B(\pi(x),(l-1)\eta)\subset B(\pi(x),d\eta)=B(\pi(x),\delta)$,
by (\ref{case2-relation1})
there is some $k_l\in \{1,\ldots,d\}$
 such that $x_{l-1}\in T^{a_{k_l}}V_0$.

Choose $\eta_l>0$ with $\eta_l<\eta$ such that for $1\leq j\leq l-1$,
\begin{align}
\label{hhh1} B( x_{l-1},\eta_l)&\subset T^{a_{k_l}}V_0,& \\
\label{hhh2} T^{c_i(n_j+\ldots+n_{l-1})} B(x_{l-1},\eta_l)&\subset T^{a_{k_{j}}}V_0,
 &\forall \;1\leq i\leq s, \\
\label{hhh3} T^{p_m(n_j+\ldots+n_{l-1})}B(x_{l-1},\eta_l)&
 \subset  T^{a_{k_{j}}}V_m ,& \forall \;1\leq m\leq t,\\
\label{hhh4} T^{q(n_j+\ldots+n_{l-1})}B(\pi(x_{l-1}),\eta_l)&
  \subset  B(\pi(x),(l-1)\eta) ,& \forall \;q\in \mathcal{C}.
\end{align}

\medskip

Let $I_l=\pi^{-1}\big(B(\pi(x_{l-1}),\eta_l)\big)$.
By (\ref{case2-relation2}),
we have $ \pi(x_{l-1})\in B(\pi(x),\delta)\subset \bigcap_{j=1}^dT^{a_j}W$
 and
\[
\pi(x_{l-1})\in \underbrace{\bigcap_{m=1}^t\pi( I_l\cap T^{a_{k_l}}V_m)\cap
\pi\big(B(x_{l-1},\eta_l)\big) }_{=:S_l} \subset \pi(I_l)=
B(\pi(x_{l-1}),\eta_l)\subset B(\pi(x_{l-1}),\eta).
\]

Let
\begin{align*}
\mathcal{A}_l& =\{p_m(n)-c_1n,\;\partial_{n_j+\ldots+n_{l-1}}p_m(n)-c_1n:1\leq m\leq t,1\leq j\leq l-1\} ,\\
\mathcal{C}_l& =\{-c_1n,\;q(n)-c_1n,\;\partial_{n_j+\ldots+n_{l-1}}q(n)-c_1n:q\in \mathcal{C},1\leq j\leq l-1\}.
\end{align*}

By our inductive hypothesis,
the conclusion of Claim \ref{ex2} holds for system $\mathcal{A}_l$ and integers $c_2-c_1,\ldots,c_s-c_1$.
Then for system $\mathcal{C}_l$ and open sets $B(x_{l-1},\eta_l),I_l\cap T^{a_{k_l}}V_1,\ldots,I_l\cap T^{a_{k_l}}V_t,\underbrace{B(x_{l-1},\eta_l),\ldots,B(x_{l-1},\eta_l)}_{t(l-1)\;\mathrm{times}}$,
there exist $y_l\in B(x_{l-1},\eta_l)$ and $n_l\in \mathbb{N}$
such that for $1\leq j\leq l-1$,
\begin{enumerate}[itemsep=4pt,parsep=2pt,label=(\arabic*)]
\item[$(la)_1$] $ T^{(c_i-c_1)n_l}y_l\in B(x_{l-1},\eta_l)$ for $2\leq i \leq s$;
\item[$(la)_2$] $ T^{p_m(n_l)-c_1n_l}y_l\in I_l\cap T^{a_{k_l}}V_m$ for $1\leq m \leq t$;
\item[$(la)_3$] $ T^{\partial_{n_j+\ldots+n_{l-1}}p_m(n_l)-c_1
   n_l}y_l\in  B(x_{l-1},\eta_l)$ for $1\leq m\leq t$;
\item[$(lc)_1$] $ T^{-c_1n_l}\pi(y_l),\;T^{q(n_l)-c_1 n_l}\pi(y_l)\in S_l$ for $q\in \mathcal{C}$;
\item[$(lc)_2$] $ T^{\partial_{n_j+\ldots+n_{l-1}}q(n_l)-c_1n_l}\pi(y_l)\in S_l\subset B(\pi(x_{l-1}),\eta_l)$
        for $q\in \mathcal{C}$.
\end{enumerate}

\medskip

Set $x_l=T^{-c_1 n_l}y_l$.
By $(la)_1$ and (\ref{hhh1}), for $1\leq i\leq s$ we have
\begin{equation}\label{AQ}
T^{c_i n_l}x_l\in B(x_{l-1},\eta_l)\subset T^{a_{k_l}}V_0.
\end{equation}
By (\ref{hhh2}) and (\ref{AQ}) , for $1\leq i\leq s,1\leq j\leq l-1$ we have
\[
T^{c_i(n_j+\ldots+n_{l-1}+n_l)}x_l\in T^{c_i(n_j+\ldots+n_{l-1})}B(x_{l-1},\eta_l)
\subset
T^{a_{k_j}}V_0 .
\]
By $(la)_2$, for $1\leq m\leq t$ we have
\[
T^{p_m(n_l)}x_l\in T^{a_{k_l}}V_m.
\]
By $(la)_3$ and (\ref{hhh3}), for $1\leq m\leq t$ and $1\leq j\leq l-1$ we have
\begin{align*}
T^{p_m(n_j+\ldots+n_{l-1}+n_l)}x_l=& T^{p_m(n_j+\ldots+n_{l-1})}(T^{\partial_{n_j+\ldots+n_{l-1}}p_m(n_l)}x_l)\ \\
   & \in T^{p_m(n_j+\ldots+n_{l-1})} B(x_{l-1},\eta_l)\subset T^{a_{k_j}}V_m.
\end{align*}
By $(lc)_1$, for $q\in \mathcal{C}$ we have
\begin{equation}\label{1212}
\pi(x_l), \;T^{q(n_l)}\pi(x_l)\in S_l\subset B(\pi(x_{l-1}),\eta)
\subset B(\pi(x),l\eta).
\end{equation}
By $(lc)_2$ and (\ref{hhh4}), for $q\in \mathcal{C}$ and $1\leq j\leq l-1$ we have
\begin{align*}
T^{q(n_j+\ldots+n_{l-1}+n_l)}\pi(x_l)=& T^{q(n_j+\ldots+n_{l-1})}(T^{\partial_{n_j+\ldots+n_{l-1}}q(n_l)}\pi(x_l))\ \\
   & \in T^{q(n_j+\ldots+n_{l-1})} B(\pi(x_{l-1}),\eta_l)\subset B(\pi(x),l\eta).
\end{align*}

There is some $k_{l+1}\in \{1,\ldots,d\}$ with $x_l\in T^{a_{k_{l+1}}}V_0$ by (\ref{case2-relation1})
and (\ref{1212}).

We finish the construction by induction.
 \end{proof}

Using Claim \ref{ex2}, we are able to give a proof of Claim \ref{ex-claim2}.

\begin{proof}[Proof of Claim \ref{ex-claim2}]
Let
\begin{align*}
\mathcal{A}_1= & \{-an^2-b_1n,\;-an^2+(c_1-b_1)n,\ldots,\;-an^2+(c_s-b_1)n\}, \\
\mathcal{C}_1= & \{q(n)-an^2-b_1n:q\in \mathcal{C}\}.
\end{align*}

Using Claim \ref{ex2} for system $\mathcal{A}_1$ and integers $b_2-b_1,\ldots,b_m-b_1$,
then for system $\mathcal{C}_1$ and open sets $W_0,W_0,W_1,\ldots,W_s$,
there exist $w\in W_0$ and $n\in \mathbb{N}$
such that
\begin{itemize}[itemsep=4pt,parsep=2pt]
\item$ T^{(b_m-b_1)n}w\in W_0$ for $2\leq m\leq t$;
\item $ T^{-an^2-b_1n}w\in W_0$;
\item $ T^{-an^2+(c_i-b_1)n}w\in W_i$ for $1\leq i\leq s$;
\item $T^{q(n)-an^2-b_1n}\pi(w)\in \bigcap_{i=0}^s\pi(W_i)$ for $q\in \mathcal{C}$.
\end{itemize}

Put $z=T^{-an^2-b_1n}w$, then the claim follows.
\end{proof}

\subsection{Proofs of Theorems \ref{polynomial-case} and \ref{polynomial-TCF}}\

Now we are able to give proofs of the main results of this section.
Proving them, we need two intermediate claims
whose proofs will be given in later subsection since they are very long.

\medskip

For a weight vector $\vec{w}=\big((\phi(w_1),w_1),\ldots,(\phi(w_k),w_k)\big)$,
define $\min (\vec{w})=w_1$ and $\max (\vec{w})=w_k$. That is,
for any system $\mathcal{A}$ with weight vector $\vec{w}$,
there are $a,b\in \mathcal{A}$
such that for all $p\in A$,
\[
\min (\vec{w})=w_1=\mathrm{deg}(a)\leq \mathrm{deg}(p) \leq \mathrm{deg}(b)= w_k=\max (\vec{w}).
\]

%The following claims play a key role in the proof of Theorem \ref{polynomial-case},

\begin{claim}\label{gene-claim1}
Assume $\pi$ has the property
$\Lambda(\mathcal{A}',\mathcal{C}')$
for any systems $\mathcal{A}'$ and $\mathcal{C}'$
if the weight vector of $\mathcal{A}'$ is $\vec{w}$. %=\big((v_1,w_1),\ldots,(v_k,w_k)\big)$.
Let $\mathcal{A}=\{p_1,\ldots,p_t\}$ be a system with weight vector $\vec{w}$,
and let $\mathcal{B}$ be a system such that $\mathrm{deg}(b)<\min(\vec{w})$ for
all $b\in \mathcal{B}$.
Then for any system $\mathcal{C}$ and open subsets $V_0,V_1,\ldots,V_t$ of $X$
with $\bigcap_{m=0}^t\pi(V_m)\neq \emptyset$,
there exist $z\in V_0$ and $n\in \mathbb{N}$
such that
\begin{itemize}[itemsep=4pt,parsep=2pt]
\item $ T^{b(n)}z\in V_0$ for $b\in \mathcal{B}$;
\item $ T^{p_m(n)}z\in V_m$ for $1\leq m\leq t$;
\item $T^{q(n)}\pi(z)\in \bigcap_{m=0}^t\pi(V_m)$ for $q\in \mathcal{C}$.
\end{itemize}
\end{claim}

\begin{claim}\label{gene-claim2}
Assume $\pi$ has the property
$\Lambda(\mathcal{A}',\mathcal{C}')$
for any systems $\mathcal{A}'$ and $\mathcal{C}'$
if the weight vector of $\mathcal{A}'$ is $\vec{w}$.
Let $\mathcal{A}$ be a system with weight vector $\vec{w}$,
and let $c_1,\ldots,c_s\in \mathcal{P}^*$ be distinct linear polynomials such that $c_i\notin \mathcal{A}$ for $1\leq i \leq s$.
Then for any system $\mathcal{C}$ and open subsets $V_0,V_1,\ldots,V_s$ of $X$
with $\bigcap_{i=0}^s\pi(V_i)\neq \emptyset$,
there exist $z\in V_0$ and $n\in \mathbb{N}$
such that
\begin{itemize}[itemsep=4pt,parsep=2pt]
\item$ T^{a(n)}z\in V_0$ for $a\in \mathcal{A}$;
\item $ T^{c_i(n)}z\in V_i$ for $1\leq i\leq s$;
\item $T^{q(n)}\pi(z)\in \bigcap_{i=0}^s\pi(V_i)$ for $q\in \mathcal{C}$.
\end{itemize}
\end{claim}

With the help of Claims \ref{gene-claim1} and \ref{gene-claim2},
we are able to show Theorem \ref{polynomial-case}.
\begin{proof}[Proof of Theorem \ref{polynomial-case} assuming Claims \ref{gene-claim1} and \ref{gene-claim2}]
For a system $\mathcal{A}$,
if $\phi(\mathcal{A})=(s,1)$ for some $s\in \mathbb{N}$,
then $\pi$ has the property $\Lambda(\mathcal{A},\mathcal{C})$ for any system $\mathcal{C}$
by Lemma \ref{linear-case-with-constraint}.

Now fix a system $\mathcal{A}$ with $(s,1)\prec \phi(\mathcal{A})$ for all $s\in \mathbb{N}$.
That is, there is some $p\in \mathcal{A}$ with $\mathrm{deg}(p)\geq 2$.
Assume $\pi$ has the property
$\Lambda(\mathcal{A}',\mathcal{C}')$
for any systems $\mathcal{A}'$ and $\mathcal{C}'$
if $\phi(\mathcal{A}')\prec\phi(\mathcal{A})$.
Fix a system $\mathcal{C}$.
We next show that $\pi$ also has the property
$\Lambda(\mathcal{A},\mathcal{C})$.

\medskip

\noindent {\bf Case 1:} $\mathrm{deg}(p)\geq2 $ for every $p\in \mathcal{A}$.

Notice that for any non-zero integer $a$,
$\partial_{a}p\neq p$ for every $p\in \mathcal{A}$.
By the similar construction in the proof of Case \ref{case1},
one can deduce that $\pi$ has the property $\Lambda(\mathcal{A},\mathcal{C})$.

\medskip

\noindent {\bf Case 2:}
There is some $p\in \mathcal{A}$ with $\mathrm{deg}(p)=1 $.

Let $\mathcal{A}=\{c_1,\ldots,c_s,\; p_1,\ldots,p_t\}$ such that $\mathrm{deg}(c_i)=1$ for $1\leq i\leq s$
and $\mathrm{deg}(p_m)\geq 2$ for $1\leq m\leq t$.
Let $V_0,V_1,\ldots,V_s,U_1,\ldots,U_t$ be open subsets of $X$ with
\[
W:= \bigcap_{i=0}^s\pi(V_i)\cap \bigcap_{m=1}^t\pi(U_m)\neq \emptyset.
\]

Let $\mathcal{A}'=\{p_1,\ldots,p_t\}$ and $\mathcal{B}=\{c_1,\ldots,c_s\}$. Then $\phi(\mathcal{A}')\prec \phi(\mathcal{A})$.
By our PET-induction hypothesis, $\pi$ has the property $\Lambda(\mathcal{A}',\mathcal{C})$.
Using Claim \ref{gene-claim1} for systems
 $\mathcal{A}'$ and $\mathcal{B}$,
then for system $\mathcal{C}$ and open sets $V_0\cap \pi^{-1}(W),U_1,\ldots,U_t$,
there exist $w\in V_0\cap \pi^{-1}(W)$ and $k\in \mathbb{N}$
such that
\begin{itemize}[itemsep=4pt,parsep=2pt]
\item$ T^{c_i(k)}w\in V_0\cap \pi^{-1}(W)$ for $1\leq i \leq s$;
\item $ T^{p_m(k)}w\in U_m$ for $1\leq m\leq t$;
\item $T^{q(k)}\pi(w)\in \pi\big(V_0\cap \pi^{-1}(W)\big)\cap \bigcap_{m=1}^t\pi(U_m)\subset W$ for $q\in \mathcal{C}$.
\end{itemize}

For every $1\leq i \leq s$,
as $T^{c_i(k)}w\in V_0\cap \pi^{-1}(W)$,
there is some $w_i\in V_i$ with $\pi(T^{c_i(k)}w)=\pi(w_i)$ which implies
\begin{equation}\label{general-equal1}
\pi(w)=\pi(T^{-c_i(k)}w_i).
\end{equation}

Choose $\gamma>0$ such that
\begin{align}
\label{gene-conA1}T^{c_i(k)} B(T^{-c_i(k)}w_i,\gamma)&\subset V_i, \;\quad\;\forall\; 1\leq i\leq s, \\
\label{gene-conA2} T^{p_m(k)}B(w,\gamma)&\subset U_m,\quad \forall \;1\leq m\leq t,\\
\label{gene-conA3} T^{q(k)}B(\pi(w),\gamma)&\subset W,\;\quad\;\forall\; q\in \mathcal{C}.
\end{align}

Let $W_0=B(w,\gamma)\cap \pi^{-1}\big( B(\pi(w),\gamma)\big)\cap V_0$
and let $W_i=B(T^{-c_i(k)}w_i,\gamma)$ for $1\leq i\leq s$.
Then $W_0$ is a non-empty open set as $w\in W_0$ and $\pi(w)\in \pi(W_i)$ by (\ref{general-equal1}).

Let
\[
\mathcal{A}''=\{\partial_kp_m:1\leq m\leq t\}\quad\text{ and} \quad
\mathcal{C}'=\{\partial_kq:q\in \mathcal{C}\}.
\]

Then $\pi$ has the property $\Lambda(\mathcal{A}'',\mathcal{C}')$
as $\phi(\mathcal{A}'')=\phi(\mathcal{A}')\prec \phi(\mathcal{A})$.
Using Claim \ref{gene-claim2} for systems
$\mathcal{A}''$ and $\mathcal{B}$, then for system
$\mathcal{C}'$ and open sets $W_0,W_1,\ldots,W_s$,
there exist $z\in W_0\subset V_0$ and $l\in \mathbb{N}$ such that
\begin{align}
\label{gene-conC2} T^{\partial_kp_m(l)}z&\in W_0\subset B(w,\gamma) ,  &\forall\;1\leq m\leq t, \\
\label{gene-conC1}T^{c_i(l)}z&\in W_i =B(T^{-c_i(k)}w_i,\gamma),&\forall\;1\leq i\leq s, \\
\label{gene-conC3}  T^{\partial_kq(l)}\pi(z)&\in \bigcap_{i=0}^s \pi(W_i)\subset \pi(W_0)\subset B(\pi(w),\gamma) ,&\forall\;q\in \mathcal{C}.
\end{align}

Recall that $\mathrm{deg}(c_i)=1$,
by (\ref{gene-conA1}) and (\ref{gene-conC1}) for $1\leq i\leq s$ we have
\[
T^{c_i(l+k)}z=T^{c_i(k)}(T^{c_i(l)}z)\in T^{c_i(k)}B(T^{-c_i(k)}w_i,\gamma)\subset V_i.
\]
By (\ref{gene-conA2}) and (\ref{gene-conC2}), for $1\leq m \leq t$ we have
\[
T^{p_m(l+k)}z=  T^{p_m(k)}(T^{\partial_kp_m(l)}z)\in T^{p_m(k)}B(w,\gamma)\subset U_m.
\]
By (\ref{gene-conA3}) and (\ref{gene-conC3}), for $q\in \mathcal{C}$ we have
\[
T^{q(l+k)}\pi(z)
=  T^{q(k)}(T^{\partial_kq(l)}\pi(z))\in T^{q(k)}B(\pi(w),\gamma)\subset W.
\]

Now set $n=l+k$, we get that
\begin{itemize}[itemsep=4pt,parsep=2pt]
\item $z\in V_0$;
\item$ T^{c_i(n)}z\in V_i$ for $1\leq i \leq s$;
\item $ T^{p_m(n)}z\in U_m$ for $1\leq m\leq t$;
\item $T^{q(n)}\pi(z)\in W=\bigcap_{i=0}^s\pi(V_i)\cap\bigcap_{m=1}^t\pi(U_m)$ for $q\in \mathcal{C}$.
\end{itemize}

This completes the proof.
\end{proof}

We are ready to show Theorem \ref{polynomial-TCF}.

\begin{proof}[Proof of Theorem \ref{polynomial-TCF}]
It follows Theorems \ref{key-thm0} and \ref{polynomial-case}.
\end{proof}

We now proceed to the proof of Claims \ref{gene-claim1} and \ref{gene-claim2}.

\begin{proof}[Proof of Claim \ref{gene-claim1}]
Fix a weight vector $\vec{w}$ with $\min(\vec{w})\geq 2$,
otherwise the system $\mathcal{B}$ will be empty and there is nothing to prove.

We prove this claim by PET-induction,
the induction on the weight vector of the system $\mathcal{B}$.
Fix a non-empty system $\mathcal{B}$ such that $\mathrm{deg}(b)<\min(\vec{w})$ for all $b\in \mathcal{B}$,
and
suppose the statement of the claim is true for any systems
$\mathcal{B}',\mathcal{A}',\mathcal{C}$ if $\phi(\mathcal{B}')\prec \phi(\mathcal{B})$ and $\phi(\mathcal{A}')=\vec{w}$.

Let $\mathcal{A}=\{p_1,\ldots,p_t\}$ be a system with weight vector $\vec{w}$, % i.e.,
%$\mathrm{deg}(p_m)\geq 2$ for every $1\leq m\leq t$,
and let $V_0,V_1,\ldots,V_t$ be open subsets of $X$
with $W:=\bigcap_{m=0}^t\pi(V_m)\neq \emptyset$.
By the similar argument in the proof of Case \ref{case1},
we may assume without loss of generality that $\pi(V_m)=W$ for $0\leq m\leq t$,
and there exist $x\in X$ with $\pi(x)\in W$, $a_1,\ldots,a_d\in \mathbb{N}$ and $\delta>0$ such that
\begin{equation}\label{general-relation1}
\pi^{-1}\big(B(\pi(x),\delta)\big)\subset  \bigcap_{m=0}^t\bigcup_{j=1}^dT^{a_j}V_m,
\end{equation}
and
\begin{equation}\label{general-relation2}
B(\pi(x),\delta)\subset W \cap \big(\bigcap_{j=1}^d  T^{a_j}W \big).
\end{equation}

Fix a system $\mathcal{C}$ and
let $\eta=\delta/d$.
Inductively we will construct $x_1,\ldots,x_d\in X,k_1,\ldots,k_{d+1}\in \{1,\ldots,d\}$ with $k_1=1$
and $n_1,\ldots,n_d\in \mathbb{N}$
such that for every $1\leq j\leq l \leq d$,
\begin{itemize}[itemsep=4pt,parsep=2pt]
\item $x_l\in T^{a_{k_{l+1}}}V_0$;
\item $\pi(x_l)\in B(\pi(x),l\eta)$;
\item $T^{b(n_j+\ldots+n_l)}x_l\in T^{a_{k_j}}V_0$ for $b\in \mathcal{B}$;
\item $T^{p_m(n_j+\ldots+n_l)}x_l\in  T^{a_{k_j}}V_m$ for $1\leq m\leq t$;
\item $T^{q(n_j+\ldots+n_l)}\pi(x_l)\in  B(\pi(x),l\eta)$ for $q\in \mathcal{C}$.
\end{itemize}
Assume this has been achieved, we can choose $1\leq j\leq l\leq d$ with $k_j=k_{l+1}$ such that
\begin{itemize}[itemsep=4pt,parsep=2pt]
\item $x_l\in T^{a_{k_{l+1}}}V_0=T^{a_{k_j}}V_0$;
\item $T^{b(n_j+\ldots+n_l)}x_l\in T^{a_{k_j}}V_0$ for $b\in \mathcal{B}$;
\item $T^{p_m(n_j+\ldots+n_l)}x_l\in T^{a_{k_j}}V_m$  for $1\leq m\leq t$;
\item $T^{q(n_j+\ldots+n_l)}\pi(x_l)\in  B(\pi(x),l\eta)
   \subset B(\pi(x),\delta)\subset \bigcap_{j=1}^d  T^{a_j}W  $  for $q\in \mathcal{C}$  \ \ \   by (\ref{general-relation2}).
\end{itemize}
Put $n=n_j+\ldots+n_l$ and $z=T^{-a_{k_{j}}}x_{l}$, then the claim follows.

\medskip

{\bf
We now return to the inductive construction of $x_1,\ldots,x_d,k_1,\ldots,k_{d+1}$ and $n_1,\ldots,n_d$.
}

\medskip

Let $b_1\in \mathcal{B}$ be an element of the minimal weight in $\mathcal{B}$.

\noindent {\bf Step 1:}
Let $I_1=\pi^{-1}\big(B(\pi(x),\eta)\big)$.
Then $I_1$ is an open subset of $X$ and
\[
\pi(x)\in \underbrace{\bigcap_{m=0}^t\pi( I_1\cap T^{a_1} V_m) }_{=:S_1}\subset \pi(I_1)=
B(\pi(x),\eta).
\]

Let
\begin{align*}
\mathcal{B}_1 &=\{b-b_1:b\in \mathcal{B}\}, \\
\mathcal{A}_1 & =\{p_m-b_1:1\leq m\leq t\}, \\
\mathcal{C}_1 &=\{-b_1,q-b_1:q\in \mathcal{C}\}.
\end{align*}

Then $\phi(\mathcal{B}_1)\prec \phi(\mathcal{B})$ by Lemma \ref{PET-induction},
and $\phi(\mathcal{A}_1)=\phi(\mathcal{A})$ as $\mathrm{deg}(b_1)<\mathrm{deg}(p_m)$ for every $1\leq m\leq t$.
By our PET-induction hypothesis,
the conclusion of Claim \ref{gene-claim1} holds for systems $\mathcal{A}_1$ and $\mathcal{B}_1$.
Then for system $\mathcal{C}_1$ and open sets $I_1\cap T^{a_1} V_0,I_1\cap T^{a_1} V_1,\ldots,I_1\cap T^{a_1} V_t$,
there exist $y_1\in I_1\cap T^{a_1}V_0$ and $n_1\in \mathbb{N}$
such that
\begin{enumerate}[itemsep=4pt,parsep=2pt,label=(\arabic*)]
\item[$(1b)$] $T^{b(n_1)-b_1 (n_1)}y_1\in I_1\cap T^{a_1} V_0$ for $b\in \mathcal{B}$;
\item[$(1a)$] $T^{p_m(n_1)-b_1(n_1)}y_1\in I_1\cap T^{a_1} V_m$ for $1\leq m\leq t$;
\item[$(1c)$] $T^{-b_1 (n_1)}\pi(y_1),\; T^{q(n_1)-b_1 (n_1)}\pi(y_1)\in S_1$ for $q\in \mathcal{C}$.
\end{enumerate}

Set $x_1=T^{-b_1 (n_1)}y_1$.
By $(1b)$, for $b\in \mathcal{B}$ we have
\[
T^{b(n_1)}x_{1}\in T^{a_1}V_0.
\]
By $(1a)$, for $1\leq m\leq t$ we have
\[
T^{p_m(n_1)}x_{1}\in T^{a_1}V_m.
\]
By $(1c)$, for $q\in \mathcal{C}$ we have
\begin{equation}\label{general0000}
\pi(x_{1}),\; T^{q(n_{1})}\pi(x_{1})\in S_1\subset  B(\pi(x),\eta) \subset  B(\pi(x),\delta).
\end{equation}

There is some $k_2\in \{1,\ldots,d\}$ with $x_1\in T^{a_{k_2}}V_0$ by (\ref{general-relation1}) and (\ref{general0000}).

\medskip

\noindent {\bf Step l:}
Let $l\geq2$ be an integer and assume that we have already chosen
$x_1,\ldots,x_{l-1}\in X$, $k_1,\ldots,k_{l-1}\in \{1,\ldots,d\}$ and $n_1,\ldots,n_{l-1}\in \mathbb{N}$ such that for $1\leq j\leq l-1$,
\begin{itemize}[itemsep=4pt,parsep=2pt]
\item $\pi(x_{l-1})\in B(\pi(x),(l-1)\eta)$;
\item $T^{b(n_j+\ldots+n_{l-1})}x_{l-1}\in T^{a_{k_{j}}}V_0$ for $b\in \mathcal{B}$;
\item $T^{p_m(n_j+\ldots+n_{l-1})}x_{l-1}\in T^{a_{k_{j}}}V_m$ for $1\leq m\leq t$;
\item $T^{q(n_j+\ldots+n_{l-1})}\pi(x_{l-1})\in  B(\pi(x),(l-1)\eta)$ for $q\in \mathcal{C}$.
\end{itemize}

As $\pi(x_{l-1})\in B(\pi(x),(l-1)\eta)\subset B(\pi(x),d\eta)=B(\pi(x),\delta)$,
by (\ref{general-relation1})
there is some $k_l\in \{1,\ldots,d\}$ such that $x_{l-1}\in T^{a_{k_l}}V_0$.

Choose $\eta_l>0$ with $\eta_l<\eta$ such that for $1\leq j\leq l-1$,
\begin{align}
\label{generalhhh1} B( x_{l-1},\eta_l)&\subset T^{a_{k_l}}V_0,& \\
\label{generalhhh2}  T^{b(n_j+\ldots+n_{l-1})}B(x_{l-1},\eta_l)&\subset
 T^{a_{k_{j}}}V_0,&\forall\; b\in \mathcal{B}, \\
\label{generalhhh3} T^{p_m(n_j+\ldots+n_{l-1})}B(x_{l-1},\eta_l)&\subset
T^{a_{k_{j}}}V_m, &\forall\; 1\leq m\leq t,\\
\label{generalhhh4} T^{q(n_j+\ldots+n_{l-1})}B(\pi(x_{l-1}),\eta_l)&\subset
 B(\pi(x),(l-1)\eta), &\forall\;q\in \mathcal{C}.
\end{align}

\medskip

Let $I_l=\pi^{-1}\big(B(\pi(x_{l-1}),\eta_l)\big)$.
By (\ref{general-relation2}),
we have $ \pi(x_{l-1})\in B(\pi(x),\delta)\subset \bigcap_{j=1}^dT^{a_j}W$
 and
\[
\pi(x_{l-1})\in \underbrace{\bigcap_{m=1}^t\pi( I_l\cap T^{a_{k_l}}V_m)\cap
\pi\big(B(x_{l-1},\eta_l)\big) }_{=:S_l} \subset \pi(I_l)=
B(\pi(x_{l-1}),\eta_l)\subset  B(\pi(x_{l-1}),\eta).
\]

Let
\begin{align*}
\mathcal{B}_l &=\{b-b_1,\;\partial_{n_j+\ldots+n_{l-1}}b-b_1:b\in \mathcal{B},1\leq j\leq l-1\}, \\
\mathcal{A}_l & = \{p_m-b_1,\;\partial_{n_j+\ldots+n_{l-1}}p_m-b_1:1\leq m\leq t,1\leq j\leq l-1\} ,\\
\mathcal{C}_l &=\{-b_1,\;q-b_1,\;\partial_{n_j+\ldots+n_{l-1}}q-b_1:q\in \mathcal{C},1\leq j\leq l-1\}.
\end{align*}

Then $\phi(\mathcal{B}_l)\prec \phi(\mathcal{B})$ by Lemma \ref{PET-induction},
and $\phi(\mathcal{A}_l)=\phi(\mathcal{A}_1)=\phi(\mathcal{A})$.
By our PET-induction hypothesis,
the conclusion of Claim \ref{gene-claim1} holds for systems $\mathcal{A}_l,\mathcal{B}_l$.
Notice that for any non-zero integer $a$,
 $\partial_{a}p_m\neq p_m$ for $1\leq m\leq t$.
Hence for system $\mathcal{C}_l$ and open sets $B(x_{l-1},\eta_l),I_l\cap T^{a_{k_l}}V_1,\ldots,I_l\cap T^{a_{k_l}}V_t,\underbrace{B(x_{l-1},\eta_l),\ldots,B(x_{l-1},\eta_l)}_{t(l-1) \; \mathrm{times}}$,
there exist $y_l\in B(x_{l-1},\eta_l)$ and $n_l\in \mathbb{N}$
such that for $1\leq j\leq l-1$,
%\footnote{Notice that $\partial_{a}p_m\neq p_m$ for any non-zero integer $a$ and $1\leq m\leq t$.}
\begin{enumerate}[itemsep=4pt,parsep=2pt,label=(\arabic*)]
\item[$(lb)_{\; \;}$]$ T^{b(n_l)-b_1(n_l)}y_l,\;
T^{\partial_{n_j+\ldots+n_{l-1}}b(n_l)-b_1(n_l)}y_l\in B(x_{l-1},\eta_l)$ for $b\in \mathcal{B}$;
\item[$(la)_1$] $ T^{p_m(n_l)-b_1(n_l)}y_l\in I_l\cap T^{a_{k_l}}V_m$ for $1\leq m \leq t$;
\item[$(la)_2$] $ T^{\partial_{n_j+\ldots+n_{l-1}}p_m(n_l)-b_1(n_l)}y_l\in  B(x_{l-1},\eta_l)$
for $1\leq m\leq t$;
\item[$(lc)_1$] $ T^{-b_1(n_l)},\; T^{q(n_l)-b_1( n_l)}\pi(y_l)\in S_l$ for $q\in \mathcal{C}$;
\item[$(lc)_2$] $ T^{\partial_{n_j+\ldots+n_{l-1}}q(n_l)-b_1(n_l)}\pi(y_l)\in S_l\subset B(\pi(x_{l-1}),\eta_l)$
for $q\in \mathcal{C}$.
\end{enumerate}

Set $x_l=T^{-b_1 (n_l)}y_l$.
By $(lb)$ and (\ref{generalhhh1}), for $b\in \mathcal{B}$ we have
\[
T^{b (n_l)}x_l\in B(x_{l-1},\eta_l)\subset T^{a_{k_l}}V_0.
\]
By $(lb)$ and (\ref{generalhhh2}), for $b\in \mathcal{B},1\leq j\leq l-1$ we have
\begin{align*}
T^{b(n_j+\ldots+n_{l-1}+n_l)}x_l=& T^{b(n_j+\ldots+n_{l-1})}(T^{\partial_{n_j+\ldots+n_{l-1}}b(n_l)}x_l)\ \\
   & \in T^{b(n_j+\ldots+n_{l-1})} B(x_{l-1},\eta_l)\subset T^{a_{k_{j}}}V_0.
\end{align*}
By $(la)_1$, for $1\leq m\leq t$ we have
\[
T^{p_m(n_l)}x_l\in T^{a_{k_l}}V_m.
\]
By $(la)_2$ and (\ref{generalhhh3}), for $1\leq m\leq t,1\leq j\leq l-1$ we have
\begin{align*}
T^{p_m(n_j+\ldots+n_{l-1}+n_l)}x_l=& T^{p_m(n_j+\ldots+n_{l-1})}(T^{\partial_{n_j+\ldots+n_{l-1}}p_m(n_l)}x_l)\ \\
   & \in T^{p_m(n_j+\ldots+n_{l-1})} B(x_{l-1},\eta_l)\subset T^{a_{k_{j}}}V_m.
\end{align*}
By $(lc)_1$, for $q\in \mathcal{C}$ we have
\begin{equation}\label{general1212}
\pi(x_l), \;T^{q(n_l)}\pi(x_l)\in S_l\subset B(\pi(x_{l-1}),\eta)
\subset B(\pi(x),l\eta).
\end{equation}
By $(lc)_2$ and (\ref{generalhhh4}), for $q\in \mathcal{C},1\leq j\leq l-1$ we have
\begin{align*}
T^{q(n_j+\ldots+n_{l-1}+n_l)}\pi(x_l)=&T^{q(n_j+\ldots+n_{l-1})}(T^{\partial_{n_j+\ldots+n_{l-1}}q(n_l)}\pi(x_l))\ \\
   & \in T^{q(n_j+\ldots+n_{l-1})} B(\pi(x_{l-1}),\eta_l)\subset B(\pi(x),l\eta).
\end{align*}

There is some $k_{l+1}\in \{1,\ldots,d\}$ with $x_l\in T^{a_{k_{l+1}}}V_0$ by (\ref{general-relation2})
and (\ref{general1212}).

We finish the construction by induction.
\end{proof}

\medskip

Using Claim \ref{gene-claim1}, we are able to give a proof of Claim \ref{gene-claim2}.

\begin{proof}[Proof of Claim \ref{gene-claim2}]
Fix a weight vector
$\vec{w}=\big((\phi(w_1),w_1),\ldots,(\phi(w_k),w_k)\big)$.
Let $\mathcal{A}$ be a system with weight vector $\vec{w}$,
and let $c_1,\ldots,c_s\in \mathcal{P}^*$ be distinct linear polynomials such that $c_i\notin \mathcal{A}$ for $1\leq i \leq s$.
Assume that $\pi$ has the property $\Lambda(\mathcal{A},\mathcal{C})$ for any system $\mathcal{C}$.

When $w_k=1$, that is, every polynomial in $\mathcal{A}$
is linear, it follows from Lemma \ref{linear-case-with-constraint}.

When $w_k\geq 2$.
Let $p\in \mathcal{A}$ with $\mathrm{deg}(p)=w_k$ and let
\[
\mathcal{A}_p=\{a\in \mathcal{A}:a,p\; \text{are equivalent}\} \quad
\mathrm{and}\quad
\mathcal{A}_r =\mathcal{A}-\mathcal{A}_p.
\]

Fix a system $\mathcal{C}$ and let
\begin{align*}
\mathcal{ B }\;& =\{a-p:a\in \mathcal{A}_{p}\}, \\
\mathcal{A}' &=\{c_i-p,\;-p,\;a-p:1\leq i\leq s,\;a\in \mathcal{A}_r\}, \\
\mathcal{C}' &= \{q-p:q\in \mathcal{C}\}.
\end{align*}
It is easy to see $\phi(\mathcal{A}')=(\phi(w_k),w_k)\prec \phi(\mathcal{A})$,
thus $\pi$ has the property $\Lambda(\mathcal{A}',\mathcal{C}')$.
For any $b\in \mathcal{B}$, there is some $a\in \mathcal{A}_p$ such that
 $\mathrm{deg}(b)=\mathrm{deg}(a-p)<\mathrm{deg}(p)=w_k=\min( \phi(\mathcal{A}'))$.
Using Claim \ref{gene-claim1} for systems
$\mathcal{A}'$ and $\mathcal{B}'$,
then for system $\mathcal{C}$
and open sets
$V_0,V_1,\ldots,V_s,\underbrace{V_0,\ldots,V_0}_{(|\mathcal{A}'|-s )\;\mathrm{times}}$,
there exist $w\in V_0$ and $n\in \mathbb{N}$
such that
\begin{itemize}[itemsep=4pt,parsep=2pt]
\item $ T^{a(n)-p(n)}w\in V_0$ for $a\in \mathcal{A}_{p}$;
\item $ T^{c_i(n)-p(n)}w\in V_i$ for $1\leq i\leq s$;
\item $ T^{-p(n)}w,\; T^{a(n)-p(n)}w\in V_0$ for $a\in \mathcal{A}_r$;
\item  $T^{q(n)-p(n)}\pi(w)\in \bigcap_{i=0}^s\pi(V_i)$ for $q\in \mathcal{C}$.
\end{itemize}

Now put $z=T^{-p(n)}w\in V_0$, then we have
\begin{itemize}[itemsep=4pt,parsep=2pt]
\item $ T^{a(n)}z\in V_0$ for $a\in \mathcal{A}_p\cup \mathcal{A}_{r}= \mathcal{A}$;
\item $ T^{c_i(n)}z\in V_i$ for $1\leq i\leq s$;
\item $T^{q(n)}\pi(z)\in \bigcap_{i=0}^s\pi(V_i)$ for $q\in \mathcal{C}$.
\end{itemize}

This completes the proof.
\end{proof}

\bibliographystyle{amsplain}

\end{document}